%% file: qt_paper.tex
\documentclass[11pt]{amsart}
\usepackage[margin = 1in]{geometry}
\usepackage{amsmath, amssymb, amsthm}
\usepackage{enumerate}
\usepackage[all]{xypic}
\usepackage{comment}
\usepackage{url}

\usepackage[usenames,dvipsnames]{xcolor}

\newtheorem{thm}{Theorem}[section]
\newtheorem{lemma}[thm]{Lemma}

\newtheorem{prop}[thm]{Proposition}
\newtheorem{cor}[thm]{Corollary}

\newtheorem{asmp}{Assumption}

\newtheoremstyle{named}{}{}{\itshape}{}{\bfseries}{.}{.5em}{#3}
\theoremstyle{named}

\theoremstyle{remark}
\newtheorem{rem}[thm]{Remark}

\makeatletter
\def\imod#1{\allowbreak\mkern5mu({\operator@font mod}\,\,#1)}
\makeatother

\let\hom\relax
\DeclareMathOperator{\hom}{Hom}

\newcommand{\Z}[0]{\mathbb{Z}}
\newcommand{\F}[0]{\mathbb{F}}
\newcommand{\C}[0]{\mathbb{C}}
\newcommand{\R}[0]{\mathbb{R}}
\newcommand{\Q}[0]{\mathbb{Q}}
\newcommand{\om}[0]{\omega}
\newcommand{\vep}[0]{\varepsilon}
\newcommand{\im}[0]{\operatorname{im}}
\newcommand{\cok}[0]{\operatorname{cok}}
\newcommand{\ord}[0]{\operatorname{ord}}
\newcommand{\gal}[0]{\operatorname{Gal}}
\newcommand{\sel}[0]{\operatorname{Sel}}
\newcommand{\res}[0]{\operatorname{res}}
\newcommand{\rk}[0]{\operatorname{rank}}
\newcommand{\rkan}[0]{\operatorname{rank}_{an}}

\newcommand{\ov}[1]{\overline{#1}}
\newcommand{\wt}[1]{\widetilde{#1}}
\newcommand{\mc}[1]{\mathcal{#1}}
\newcommand{\wh}[1]{\widehat{#1}}
\newcommand{\gen}[1]{\langle #1 \rangle}
\newcommand{\abs}[1]{\left| #1 \right|}

\newcommand{\smat}[4]{\left(\begin{smallmatrix} #1 & #2 \\#3 & #4\end{smallmatrix}\right)}

\title{Quadratic Twists of Elliptic Curves with 3-Selmer Rank 1}
\author{Zane Kun Li}

\address{Department of Mathematics, UCLA, Los Angeles, CA 90095-1555}
\email{zkli@math.ucla.edu}


\begin{document}
\begin{abstract}
A weaker form of a 1979 conjecture of Goldfeld states that for every elliptic curve $E/\Q$,
a positive proportion of its quadratic twists $E^{(d)}$ have rank 1.
Using tools from Galois cohomology, we give criteria on $E$ and $d$ which force a positive proportion of the quadratic twists of $E$ to have
3-Selmer rank 1 and global root number $-1$. We then give four nonisomorphic infinite families
of elliptic curves $E_{m, n}$ which satisfy these criteria.
Conditional on the rank part of the Birch and Swinnerton-Dyer conjecture,
this verifies the aforementioned conjecture for infinitely many elliptic curves.
Our elliptic curves are easy to give explicitly and we state precisely which quadratic
twists $d$ to use. Furthermore, our methods
have the potential of being generalized to elliptic curves over other number fields.
\end{abstract}

\maketitle

\input{qt_intro.tex}
\input{qt_selmer.tex}
\input{qt_example.tex}

\subsubsection*{Acknowledgements}
The author would like to thank his senior thesis advisor Professor Christopher Skinner for his guidance and helpful discussions.
The author would also like to thank the referee for the valuable suggestions on improving this paper.

\bibliographystyle{amsalpha}
\bibliography{qt_ref}

\end{document}

%% file: qt_intro.tex
\section{Introduction}
Let $E$ be an elliptic curve over $\Q$. For a squarefree integer $d \neq 0$,
if $E$ is given by the short Weierstrass form $y^{2} = x^{3} + ax + b$, then
the quadratic twist $E^{(d)}/\Q$ is given by $dy^{2} = x^{3} + ax + b$.
Note that $E^{(d)}$ is the unique elliptic curve (up to $\Q$-isomorphism) isomorphic to $E$ over $\Q(\sqrt{d})$.
In 1979, Goldfeld \cite{goldfeld} conjectured that for any elliptic curve $E/\Q$,
$$\sum_{\substack{\abs{d} \leq X\\\mu(d)^{2} = 1}} \rk(E^{(d)}) \sim \frac{1}{2}\sum_{\substack{\abs{d} \leq X\\\mu(d)^{2} = 1}} 1$$
as $X \rightarrow \infty$.
Assuming both the Parity Conjecture and Goldfeld's Conjecture, this implies that for each fixed $E/\Q$, half the quadratic twists
of $E$ should be of rank 0 and the other half rank 1. A weaker form of Goldfeld's conjecture is that for $r = 0$ or $1$,
\begin{align*}
N_{r}(X) := \#\{\text{squarefree } d \in \Z : \abs{d} \leq X, \rk(E^{(d)}) = r\} \gg X
\end{align*}
as $X \rightarrow \infty$. That is, a positive proportion of quadratic twists of $E$ have rank $r$. It is this weaker form of Goldfeld's Conjecture that we study.

In 1998, Ono and Skinner in \cite{os98} using results
of Waldspurger and of Friedberg and Hoffstein showed that $N_{0}(X) \gg X/\log X$ for all elliptic curves $E/\Q$.
This is currently the best known general result.
Vatsal in \cite{v97} showed that $N_{0}(X) \gg X$ for any semistable elliptic curve $E/\Q$ with
a rational point of order 3 and good ordinary reduction at 3.

For the $r = 1$ case,
the best known general result is by Perelli and Pomykala in \cite{pp97} who showed that $N_{1}(X) \gg_{\vep} X^{1 - \vep}$ for all $E/\Q$.
Up until 2009, $N_{1}(X) \gg X$ was only known unconditionally for two elliptic curves
$y^{2} + y = x^{3} + x^{2} - 9x - 15$
in \cite{v98} and
$y^{2}+ y = x^{3} + x^{2} - 23x - 50$
in \cite{byeon04}.
However in 2009, Byeon et al. in \cite{bjk} showed that $N_{1}(X) \gg X$ for infinitely many elliptic curves over $\Q$.
That is, they show that by the solution to a variant of the binary Goldbach problem for polynomials, there are infinitely many
integers $m$ such that $8(9m+4)^{3} = 27p + q$ for some primes $p(\neq 3)$ and $q$. Then for such $m, p$, and $q$ (for example such a triple
could be $(m, p, q) = (1, 7, 17387)$), Byeon et al.
are able to show that the optimal elliptic curve of the isogeny class of
$y^{2} + 2(9m + 4)xy + py = x^{3}$ is such that $N_{1}(X) \gg X$. Since there are infinitely many such
$m$, this yields infinitely many elliptic curves such that $N_{1}(X) \gg X$.
Previous work (\cite{bjk, by11}) have used tools such as Dedekind eta products
to approach this problem. We shall approach this problem differently, that is, from the point of view of Selmer groups and Galois cohomology.

From now on, unless otherwise stated, $E/\Q$ will be always given by the global minimal Weierstrass equation
and hence the primes dividing the discriminant $\Delta_{E}$ are exactly the primes of bad reduction and the conductor $N_{E} \mid \Delta_{E}$.

%
%
\begin{thm}\label{resthm}
Fix a semistable elliptic curve $E/\Q$ such that 3 is of good reduction, $\ord_{2}(\Delta_{E})$ is odd, $\ord_{2}(N_{E}) = 1$, $3 \nmid \ord_{v}(\Delta_{E})$
for all places $v$ of bad reduction, and $E = E'/\langle P \rangle$ where $P$ is a rational 3-torsion point of the elliptic curve $E'/\Q$.

Suppose we choose quadratic twists $d$ as follows. Let $d$ be a squarefree positive integer such that $(d, \Delta_{E}) = 1$,
$3 \nmid h(\Q(\sqrt{d}))$ where $h(\Q(\sqrt{d}))$ is the class number of $\Q(\sqrt{d})$,
$d \equiv 1 \imod{3}$, $d \equiv 1 \imod{4}$, and for each prime $\ell \mid N_{E}$, we require $d$ as follows:
\begin{align*}
\begin{array}{ccc|c}
\text{Hypotheses on $\ell$} &&& \text{Choice for $d$}\\\hline
\text{$E$ has split reduction at $\ell$} & \ell = 2 && d \equiv 5 \imod{8}\\
& \ell \neq 2& \ell \equiv 1 \imod{3} & d \not\equiv \square \imod{\ell}\\
& \ell \neq 2& \ell \equiv 2 \imod{3} & d \equiv 1 \imod{\ell}\\
\text{$E$ has nonsplit reduction at $\ell$} & \ell = 2 && d \equiv 1 \imod{8}\\
& \ell \neq 2& \ell \equiv 1 \imod{3} & d \equiv 1 \imod{\ell}\\
& \ell \neq 2& \ell \equiv 2 \imod{3} & d \not\equiv \square \imod{\ell}.
\end{array}
\end{align*}
Then $d$ satisfies a congruence mod $12N_{E}$ and $3 \nmid h(\Q(\sqrt{d}))$.

Finally, suppose for all $d$ chosen above, $E$ is such that the global root number $\om(E^{(d)}) = -1$ and the torsion subgroup $E^{(d)}(\Q)_{\textrm{tors}} = 0$.
Then $\sel_{3}(E^{(d)}/\Q) = \Z/3\Z$.
Furthermore, a positive proportion of the quadratic twists $E^{(d)}/\Q$ have 3-Selmer group equal to $\Z/3\Z$ and global root number $-1$. In particular,
such twists are precisely given by the $d$ mentioned above.
\end{thm}
%

In the lemmas and propositions leading to the proof of Theorem \ref{resthm}, we have mentioned precisely which assumptions on $E$ and $d$ we use.
Most of the assumptions above are technical assumptions on $E$ or $d$ that help us simplify the analysis of the 3-Selmer group since we are just aiming
to find a positive proportion of quadratic twists.
We however remark that the congruence conditions resulting in a congruence mod $12N_{E}$ is used in Section \ref{infd} to show that the quadratic twist has almost everywhere unramified
3-Selmer group. Furthermore, the assumption that for all chosen $d$, $E$ is such that the global root number $\om(E^{(d)}) = -1$ and the torsion subgroup $E^{(d)}(\Q)_{\textrm{tors}} = 0$
is used in the proof of Proposition \ref{3selrank} so that we can consider instead $\sel_{3^{\infty}}(E^{(d)}/\Q)$ and use the $p$-parity conjecture.

\begin{cor}\label{rescor}
Assume that the rank part of the Birch and Swinnerton-Dyer conjecture is true. If $E$ is as in Theorem \ref{resthm}, then $E$ is an elliptic curve
satisfying $N_{1}(X) \gg X$.
\end{cor}
We now give infinitely many elliptic curves satisfying the assumptions above.
Let $m = 1, 7, 13, 19$ and $n$ such that
\begin{align}\label{choice}
\mp 2(m + 24n)(62208n^{2} + (5184m \mp 432)n + (108m^{2} \mp 18m + 1))
\end{align}
is squarefree.\footnote{We shall use the following convention throughout this paper: Choose either $-$ in $\mp$
or $+$ in $\mp$ for \eqref{choice}. If one chose $-$, then all other expressions with $\mp$ are to be interpreted
as $-$ and all other expressions with $\pm$ are to be interpreted as $+$. Similarly, if
one instead chose $+$.}
By a theorem of Erd\"{o}s, there are infinitely many such $n$.
Define $A = 18(m + 24n) \mp 1$, $B = \pm 4/9$, $D = -3$, and $r = 1/3 + A^{2}$.
Let $$E_{m, n}/\Q \colon y^{2} + xy = x^{3} + H(m, n)x + J(m, n)$$
where $H(m, n) = (2ABD + 2A^{2}Dr + 3r^{2})/16$ and $J(m, n) = (B^{2}D + 2ABDr + A^{2}Dr^{2} + r^{3})/64$.
As we show in Lemmas \ref{exlem1} and \ref{exlem2}, both $H(m, n)$ and $J(m, n)$ are integers.

\begin{thm}\label{exthm}
For each $m = 1, 7, 13, 19$ and $n$ as in \eqref{choice}, $E_{m, n}/\Q$ is an elliptic curve with a positive proportion
of its quadratic twists having 3-Selmer group $\Z/3\Z$ and global root number $-1$.
\end{thm}
Assuming the Birch and Swinnerton-Dyer conjecture, this result proves using purely algebraic methods that
infinitely many elliptic curves satisfy $N_{1}(X) \gg X$. Furthermore, from our methods, we know precisely which quadratic twists
are of rank 1 and our elliptic curves are straightforward to construct since we only need \eqref{choice} to be squarefree.
This gives, albeit conditionally, another answer to Problem 9.33 of \cite{ono}.

Since we make extensive use of Selmer groups and Galois cohomological methods, it seems likely that our methods can be extended
to other number fields, provided that a similar Davenport-Heilbronn/Nakagawa-Horie result on the 3-rank of class groups of quadratic fields
exists over the given number field
to show the positive proportion of $d$. It would be interesting to see if our methods can be adapted to work
over $\Q(\sqrt{-1})$ or $\Q(\sqrt{5})$.

While Corollary \ref{rescor} assumes the Birch and Swinnerton-Dyer conjecture is true, we note that one can most likely
combine \cite[Theorem 2.10]{v97} and the $p$-adic Gross-Zagier formula of Perrin-Riou to show the result unconditionally.

This paper is organized as follows.
In Section \ref{infd}, we show that for certain $d$ satisfying a congruence condition depending on $E$, the quadratic twist $E^{(d)}$
has an almost everywhere unramified 3-Selmer group.
Carefully analyzing the 3-Selmer group along with further conditions on $d$ and
a consideration of the $3^{\infty}$-Selmer group
culminate in allowing us to determine the 3-Selmer group in Proposition \ref{3selrank} at the end of Section \ref{furthercond}.
We prove Theorem \ref{resthm} and Corollary \ref{rescor} in Section \ref{respfs}.
In Section \ref{example}, we give our explicit example $E_{m, n}$ and show that it satisfies all our desired conditions, thus proving Theorem \ref{exthm}.

It should be noted that
Section \ref{selmer} is similar to the work of \cite{wang} who uses Galois cohomological methods
to show that certain elliptic curves have a positive proportion of its quadratic twists with trivial 3-Selmer group (and hence rank 0). Wang mentions at the
end of his thesis that one could most likely use similar methods to prove an analogous result for quadratic twists with 3-Selmer group $\Z/3\Z$. It is this method
that we follow. Certain results remain the same as in \cite{wang}, however for completeness, we have included their proofs.

%% file: qt_selmer.tex
\section{Proof of Theorem \ref{resthm}}\label{selmer}
Throughout this section, we will make various assumptions to show that for certain elliptic curves, a positive proportion of its quadratic twists
have 3-Selmer group $\Z/3\Z$ and global root number $-1$.
The format of our assumptions will be as follows. Each assumption will contain two parts, part $(i)$ will be used for assumptions on our elliptic curve
$E$ and part $(ii)$ will be used for assumptions on $d$ once $E$ is fixed. We make our first assumption below.

\begin{asmp}\label{asmp1}
Assume that:
\begin{enumerate}[$(i)$]
\item $E/\Q$ is semistable, and
\item $d$ is squarefree, $d \equiv 1 \imod{4}$, and $(d, \Delta_{E}) = 1$.
\end{enumerate}
\end{asmp}
Recall that the discriminant of the quadratic twist $E^{(d)}$ is such that $\Delta_{E^{(d)}} = d^{6}\Delta_{E}$.
The assumption that $d \equiv 1 \imod{4}$ implies that $N_{E^{(d)}} = d^{2}N_{E}$. As $(d, \Delta_{E}) = 1$, $(d, N_{E}) = 1$ since the conductor divides the discriminant.
Recall that $E$ has good, multiplicative, or additive reduction at a prime $\ell$ if and only if $\ord_{\ell}(N_{E})$ is $0$, $1$, or $\geq 2$.
Since $(d, N_{E}) = 1$ and $N_{E^{(d)}} = d^{2}N_{E}$, $E^{(d)}$ has additive reduction at all primes dividing $d$, multiplicative reduction at all
primes dividing $N_{E}$, and good reduction at all other primes.

\subsection{Preliminaries}
We say $E$ and $E^{(d)}$ have the ``same multiplicative splitting type" at a prime $\ell$ when $E$ and $E^{(d)}$ either both have split multiplicative
reduction or both have nonsplit multiplicative reduction at $\ell$.
\begin{prop}[\cite{wang}, Proposition 3.1]\label{prop1}
Suppose $\ell \neq 3$ is a prime of multiplicative
reduction for $E$ (and hence also for $E^{(d)}$ by Assumption \ref{asmp1}).
Then $E/\Q_{\ell}$ and $E^{(d)}/\Q_{\ell}$ have the same multiplicative splitting type if and only if
$d$ is a square in $\Q_{\ell}$.
\end{prop}
\begin{rem}
We give an algebraic proof below. A more explicit proof can be given
by using \cite[Proposition 4.4]{schmitt} which states that $\ell \neq 2, 3$ is split multiplicative if and only if
$-c_{4}c_{6}$ is a square mod $\ell$; $\ell = 3$ is split multiplicative if and only if $b_{2}$ is a square mod $\ell$;
and $\ell = 2$ is split multiplicative if and only if $x^{2} + a_{1}x + (a_{3}a_{1}^{-1} + a_{2})$ has a root in $\F_{2}$.
This more algebraic proof below however has an added advantage of giving Corollary \ref{cor2}.
\end{rem}
\begin{proof}
If $d$ is a square of $\Q_{\ell}$, then $\Q_{\ell}(\sqrt{d}) = \Q_{\ell}$ and hence $E$ and $E^{(d)}$ are isomorphic over $\Q_{\ell}$.
Then their reductions to $\F_{\ell}$ are isomorphic which implies that $E/\Q_{\ell}$ and $E^{(d)}/\Q_{\ell}$ have the same multiplicative splitting type.

Now suppose $d$ is not a square of $\Q_{\ell}$. As $E/\Q_{\ell}$ and $E^{(d)}/\Q_{\ell}$
are isomorphic over $\Q_{\ell}(\sqrt{d})$, we may identify the $\Q_{\ell}$-rational points of
$E$ and $E^{(d)}$, $E(\Q_{\ell})$ and $E^{(d)}(\Q_{\ell})$ respectively, as subgroups of $E(\Q_{\ell}(\sqrt{d}))$ fixed by the ordinary
Galois action $\varphi_{\sigma}: P \mapsto P^{\sigma}$, $\sigma \in \gal(\ov{\Q}/\Q)$, and the twisted Galois action $\psi_{\sigma}: P \mapsto [\chi_{d}(\sigma)] P^{\sigma}$ (where we have used the quadratic character
$\chi_{d}$ to define a 1-cocycle, see \cite[p. 321]{silverman1}). This gives a map
$\phi: E(\Q_{\ell}) \oplus E^{(d)}(\Q_{\ell}) \rightarrow E(\Q_{\ell}(\sqrt{d}))$ defined such that $P \oplus Q \mapsto P + Q$.
Therefore $\ker\phi = \{P \oplus (-P): P \in E(\Q_{\ell}) \cap E^{(d)}(\Q_{\ell})\}$. For $P \in \ker\phi$, $P$ is fixed by both the ordinary and twisted
Galois actions. As $P = \varphi_{\sigma}(P) = P^{\sigma}$ and $P = \psi_{\sigma}(P) = [\chi_{d}(\sigma)]P^{\sigma} = [\chi_{d}(\sigma)]P$, since there
exists an $\sigma$ with $\chi_{d}(\sigma) \neq 1$, we have $P = -P$. Therefore $\ker\phi \cong E(\Q_{\ell})[2]$.

Let $c$ be a generator $\gal(\Q_{\ell}(\sqrt{d})/\Q_{\ell})$. For any $P \in E(\Q_{\ell}(\sqrt{d}))$, we have $(P + P^{c})^{c} = P + P^{c}$
and $(P - P^{c})^{c} = [\chi_{d}(c)](P - P^{c})$. Indeed, writing $P = (x, y)$, we have $P^{c} = (x, \chi_{d}(c)y)$ and hence $-P^{c} = (x, -\chi_{d}(c)y)$.
Then as $\chi_{d}(c)^{2} = 1$, $(P + P^{c})^{c} = P^{c} + P$ and
\begin{align*}
(P - P^{c})^{c} &= P^{c} + (-P^{c})^{c} = (x, \chi_{d}(c)y) + (x, -y)\\& = [\chi_{d}(c)]((x, y) + (x, -\chi_{d}(c)y)) = [\chi_{d}(c)](P - P^{c}).
\end{align*}
This implies that $P + P^{c} \in E(\Q_{\ell})$ and $P - P^{c} \in E^{(d)}(\Q_{\ell})$. Therefore
\begin{align*}
2P = (P + P^{c}) + (P - P^{c}) = \phi((P + P^{c}) \oplus (P - P^{c})) \in \im\phi
\end{align*}
for every $P \in E(\Q_{\ell}(\sqrt{d}))$. Then $E(\Q_{\ell}(\sqrt{d}))/2E(\Q_{\ell}(\sqrt{d})) \twoheadrightarrow E(\Q_{\ell}(\sqrt{d}))/\im\phi$, that
is, $\cok\phi$ is a quotient group of $E(\Q_{\ell}(\sqrt{d}))/2E(\Q_{\ell}(\sqrt{d}))$. Thus $\cok\phi$ and $\ker\phi$ are finite 2-groups.

The same argument as above can be applied to $\wt{E}_{ns}(\F_{\ell})$ and $\wt{E}_{ns}^{(d)}(\F_{\ell})$, the group of nonsingular reduced points of $E(\Q_{\ell})$ and $E^{(d)}(\Q_{\ell})$.
Viewing these groups as fixed points on $\wt{E}_{ns}(\F_{\ell^{2}})$ (since $d$ is not a square in $\Q_{\ell}$, $\Q_{\ell}(\sqrt{d})$ reduces to $\F_{\ell^{2}}$)
under the ordinary and twisted Galois actions respectively, we have the short exact sequence
\begin{align*}
0 \longrightarrow \ker\wt{\phi} \longrightarrow \wt{E}_{ns}(\F_{\ell}) \oplus \wt{E}_{ns}^{(d)}(\F_{\ell}) \overset{\wt{\phi}}{\longrightarrow} \wt{E}_{ns}(\F_{\ell}(\sqrt{d})) \longrightarrow \cok\wt{\phi} \longrightarrow 0.
\end{align*}
Since all the groups above are finite, we have
\begin{align*}
\# \ker\wt{\phi} \cdot \# \wt{E}_{ns}(\F_{\ell}(\sqrt{d})) = \#(\wt{E}_{ns}(\F_{\ell}) \oplus \wt{E}_{ns}^{(d)}(\F_{\ell})) \cdot \#\cok\wt{\phi}.
\end{align*}
That is,
\begin{align*}
\# \wt{E}_{ns}(\F_{\ell}(\sqrt{d})) = \frac{\#\cok\wt{\phi}}{\#\ker\wt{\phi}} \cdot \# \wt{E}_{ns}(\F_{\ell}) \cdot \# \wt{E}_{ns}^{(d)}(\F_{\ell}).
\end{align*}
Note that $\#\cok\wt{\phi}/\#\ker\wt{\phi}$ is a nonnegative power of 2.
Since $\Q_{\ell}(\sqrt{d})/\Q_{\ell}$ is a finite extension, if $E$ has multiplicative reduction over $\Q_{\ell}$, then it also has
multiplicative reduction over $\Q_{\ell}(\sqrt{d})$.
Then $\# \wt{E}_{ns}(\F_{\ell}(\sqrt{d})) = \# \wt{E}_{ns}(\F_{\ell^{2}}) = \ell^{2} \pm 1$.
Note that both $\wt{E}_{ns}(\F_{\ell})$ and $\wt{E}_{ns}^{(d)}(\F_{\ell})$ have cardinality $\ell \pm 1$ (not necessarily both the same cardinality).
If $\ell = 2$, the cardinalities of $\wt{E}_{ns}(\F_{\ell}(\sqrt{d}))$,
$\wt{E}_{ns}(\F_{\ell})$, and $\wt{E}_{ns}^{(d)}(\F_{\ell})$ are all prime to 2 and hence $\#\cok\wt{\phi}/\#\ker\wt{\phi} = 1$. For $\ell \geq 5$,
$\ell^{2} - 1 > (\ell + 1)^{2}/2$ and $\ell^{2} + 1 < 2(\ell - 1)^{2}$ and hence $\#\cok\wt{\phi}/\#\ker\wt{\phi} = 1$. Therefore
we must have $$\# \wt{E}_{ns}(\F_{\ell}(\sqrt{d})) = \# \wt{E}_{ns}(\F_{\ell}) \cdot \# \wt{E}_{ns}^{(d)}(\F_{\ell})$$ which implies that
$\# \wt{E}_{ns}(\F_{\ell}(\sqrt{d})) = \ell^{2} - 1$ and one of $\# \wt{E}_{ns}(\F_{\ell})$ and $\wt{E}_{ns}^{(d)}(\F_{\ell})$ is
$\ell - 1$ and the other is $\ell + 1$. That is, $E/\Q_{\ell}(\sqrt{d})$ has split multiplicative reduction and
exactly one of $E/\Q_{\ell}$ and $E^{(d)}/\Q_{\ell}$ has split multiplicative reduction, and the other has nonsplit multiplicative reduction.
This completes the proof of Proposition \ref{prop1}.
\end{proof}

The proof of Proposition \ref{prop1} also yields the following result.
\begin{cor}\label{cor2}
Let $E/K$ be an elliptic curve and let $L/K$ be a quadratic extension and $E^{L}/K$ the quadratic twists of $E$ with
respect to $L$. Then viewing $E(K)$ and $E^{L}(K)$ as subgroups fixed by the ordinary and twisted Galois actions
of $\gal(L/K)$ on $E(L)$, the map $\phi: E(K) \oplus E^{L}(K) \rightarrow E(L)$ with $P \oplus Q \mapsto P + Q$
has the property that both its kernel and cokernel are finite 2-groups.
\end{cor}
\begin{proof}
The proof is exactly as starting in the second paragraph of the proof of Proposition \ref{prop1} except we replace
$\Q_{\ell}$ with $K$ and $\Q_{\ell}(\sqrt{d})$ with $L$.
\end{proof}

In the next section, we will make use of the following general lemma which follows from the Snake Lemma.
\begin{lemma}[\cite{wang}, Lemma 3.4]\label{lem1}
If $\phi: A \rightarrow B$ has an $n$-divisible kernel and cokernel, and the cokernel has no $n$-torsion, then $A/nA \cong B/nB$.
\end{lemma}
\begin{proof}
We have the following commutative diagram where both rows are short exact sequences:
\begin{align*}
\xymatrix{
0 \ar[r] &\ker\phi \ar[r]\ar[d]^{\times n} &A \ar[r]^-{\phi}\ar[d]^{\times n} &\im\phi \ar[r]\ar[d]^{\times n} &0\\
0 \ar[r] &\ker\phi \ar[r] &A \ar[r]^-\phi &\im\phi \ar[r] &0
}
\end{align*}
By the Snake Lemma, since both rows are exact, we have the following exact sequence 
\begin{align*}
\xymatrix{
\cdots \ar[r] &(\im\phi)[n] \ar[r] &\ker\phi/n\ker\phi \ar[r] &A/nA \ar[r]^-{\phi} &\im\phi/n\im\phi \ar[r] &0.
}
\end{align*}
Since $\phi$ has an $n$-divisible kernel, $\ker\phi/n\ker\phi = 0$ and hence $A/nA \cong \im\phi/n\im\phi$.
Similarly considering $B$, we have
\begin{align*}
\xymatrix{
0 \ar[r] &\im\phi \ar[r]\ar[d]^{\times n} &B \ar[r]^-{\phi}\ar[d]^{\times n} &\cok\phi \ar[r]\ar[d]^{\times n} &0\\
0 \ar[r] &\im\phi \ar[r] &B \ar[r]^-{\phi} &\cok\phi \ar[r] &0
}
\end{align*}
which again applying the Snake Lemma yields that
\begin{align*}
\xymatrix{
\cdots \ar[r] &(\cok\phi)[n] \ar[r] &\im\phi/n\im\phi \ar[r] &B/nB \ar[r]^-{\phi} &\cok\phi/n\cok\phi \ar[r] &0.
}
\end{align*}
Since $\phi$ has an $n$-divisible cokernel and no $n$-torsion, $\cok\phi/n\cok\phi = 0$ and $(\cok\phi)[n] = 0$.
This implies that $\im\phi/n\im\phi \cong B/nB$ and hence $A/nA \cong B/nB$.
This completes the proof of Lemma \ref{lem1}.
\end{proof}

%

\subsection{Congruence conditions for $d$}\label{infd}
We now show that for infinitely many squarefree $d$, the twist $E^{(d)}$ of certain elliptic curves $E$ have an almost
everywhere unramified 3-Selmer group.
We will use the following notational convention. For $K$ a number field, $M$ a $\gal(\ov{K}/K)$-module, and $\Sigma$
a subset of all places of $K$, define $H^{1}(K, M; \Sigma)$ to be the elements of $H^{1}(K, M)$ unramified at all places
outside $\Sigma$.

\begin{asmp}\label{asmp2}
Assume that:
\begin{enumerate}[$(i)$]
\item For $E/\Q$, $\ord_{2}(\Delta_{E})$ is odd, 3 is a prime of good reduction, and at all places $v$ of bad reduction, $3 \nmid \ord_{v}(\Delta_{E})$.
\item For primes $\ell$ of multiplicative reduction for $E^{(d)}$ (that is for primes $\ell \mid N_{E}$),
depending on whether or not $\ell$ is of split or nonsplit multiplicative reduction for $E$
and whether $\ell \equiv 1$ or $2 \imod{3}$ we make the following choice of $d$:
\begin{align}\label{dtab}
\begin{array}{ccc|c}
\text{Hypotheses on $\ell$} &&& \text{Choice for $d$}\\\hline
\text{$E$ has split reduction at $\ell$} & \ell = 2 && d \not\equiv 1 \imod{8}\\
& \ell \neq 2& \ell \equiv 1 \imod{3} & d \not\equiv \square \imod{\ell}\\
& \ell \neq 2& \ell \equiv 2 \imod{3} & d \equiv \square \imod{\ell}\\
\text{$E$ has nonsplit reduction at $\ell$} & \ell = 2 && d \equiv 1 \imod{8}\\
& \ell \neq 2& \ell \equiv 1 \imod{3} & d \equiv \square \imod{\ell}\\
& \ell \neq 2& \ell \equiv 2 \imod{3} & d \not\equiv \square \imod{\ell}.
\end{array}
\end{align}
\end{enumerate}
\end{asmp}
\begin{rem}
Note that our choice of $d$ in the case of $\ell = 2$ is based on the fact that $p$ is a square in $\Q_{2}$ if and only if $p \equiv 1 \imod{8}$.
\end{rem}
\begin{lemma}\label{lem2}
Let $\Sigma$ be a finite set of places containing the archimedean places, finite places where $E^{(d)}$ has
bad reduction, and the prime 3. Then the 3-Selmer group $\sel_{3}(E^{(d)}/\Q) \subset H^{1}(\Q, E^{(d)}[3]; \Sigma)$.
\end{lemma}
\begin{proof}
This immediately follows from an application of \cite[Corollary X.4.4]{silverman1}.
\end{proof}
\begin{prop}\label{prop3}
Under Assumptions \ref{asmp1} and \ref{asmp2},
$\sel_{3}(E^{(d)}/\Q) \subset H^{1}(\Q, E^{(d)}[3]; \{3\}).$
\end{prop}
\begin{proof}
Lemma \ref{lem2} implies that it suffices to show that $\sel_{3}(E^{(d)}/\Q)$ is unramified at the archimedean place
and the places where $E^{(d)}$ has bad reduction.

We first consider the archimedean place. Corollary \ref{cor2} applied to $L = \C$ and $K = \R$ shows that the map
$\phi: E^{(d)}(\R) \oplus (E^{(d)})^{\C}(\R) \rightarrow E^{(d)}(\C)$ has finite 2-group kernel and cokernel.
Then Lemma \ref{lem1} implies that
$$E^{(d)}(\R)/3E^{(d)}(\R) \oplus (E^{(d)})^{\C}(\R)/3(E^{(d)})^{\C}(\R) \cong E^{(d)}(\C)/3E^{(d)}(\C).$$
Since $\C$ is algebraically closed, $E^{(d)}(\C)/3E^{(d)}(\C) = 0$. Then the above isomorphism implies that
$E^{(d)}(\R)/3E^{(d)}(\R) = 0$ and hence $\sel_{3}(E^{(d)}/\Q)$ is unramified at the archimedean place.

We now consider the places where $E^{(d)}$ has bad reduction. By Assumption \ref{asmp1}, $E^{(d)}$
has additive reduction at all primes $\ell$ dividing $d$ and multiplicative reduction at all primes dividing $N_{E}$.

Consider the primes of additive reduction for $E^{(d)}$, that is primes $\ell \mid d$.
By Assumption \ref{asmp1}, since $(d, \Delta_{E}) = 1$, $\ell$ is of good reduction for $E$.
Then $E/\Q_{\ell}(\sqrt{d})$ also has good reduction. By Corollary \ref{cor2} the map
$\phi: E(\Q_{\ell}) \oplus E^{(d)}(\Q_{\ell}) \rightarrow E(\Q_{\ell}(\sqrt{d}))$ has finite 2-group kernel and cokernel. An application
of Lemma \ref{lem1} yields that
\begin{align}\label{prop3eq1a}
E(\Q_{\ell})/3E(\Q_{\ell}) \oplus E^{(d)}(\Q_{\ell})/3E^{(d)}(\Q_{\ell}) \cong E(\Q_{\ell}(\sqrt{d}))/3E(\Q_{\ell}(\sqrt{d})).
\end{align}
Since $\ell$ is of good reduction for $E$, we have the exact sequence
\begin{align*}
\xymatrix{0 \ar[r] &E_{1}(\Q_{\ell}) \ar[r] &E(\Q_{\ell}) \ar[r] & \wt{E}(\F_{\ell}) \ar[r] & 0.}
\end{align*}
Since 3 is a prime of good reduction for $E$ and $E_{1}(\Q_{\ell})/3E_{1}(\Q_{\ell}) = \wh{E}(\mc{M})/3\wh{E}(\mc{M}) = 0$
where $\wh{E}(\mc{M})$ is the formal group associated with $E$ and $\mc{M}$ is the maximal ideal of $\Z_{\ell}$, by the exactness of the above sequence and
Lemma \ref{lem1}, we have $E(\Q_{\ell})/3E(\Q_{\ell}) \cong \wt{E}(\F_{\ell})/3\wt{E}(\F_{\ell}).$ We now repeat the above with $\Q_{\ell}$
replaced by $\Q_{\ell}(\sqrt{d})$. Since $\ell \mid d$, $\Q_{\ell}(\sqrt{d})$ is a totally ramified extension of $\Q_{\ell}$ and hence
the residue field of $\Q_{\ell}(\sqrt{d})$ is $\F_{\ell}$. Then by the same reasoning as above, we have the isomorphism
$E(\Q_{\ell}(\sqrt{d}))/3E(\Q_{\ell}(\sqrt{d})) \cong \wt{E}(\F_{\ell})/3\wt{E}(\F_{\ell})$ and hence
$$E(\Q_{\ell}(\sqrt{d}))/3E(\Q_{\ell}(\sqrt{d})) \cong E(\Q_{\ell})/3E(\Q_{\ell}).$$
Combining this with \eqref{prop3eq1a} yields that $E^{(d)}(\Q_{\ell})/3E^{(d)}(\Q_{\ell}) = 0$ and hence
$\sel_{3}(E^{(d)}/\Q)$ is unramified at all places $\ell \mid d$.

We now consider the primes of multiplicative reduction for $E^{(d)}$, that is primes $\ell \mid N_{E}$. Fix such an $\ell$.
We have the filtration $E_{1}^{(d)}(\Q_{\ell}) \subset E_{0}^{(d)}(\Q_{\ell}) \subset E^{(d)}(\Q_{\ell})$. Recall that we
also have $E_{1}^{(d)}(\Q_{\ell}) \cong \wh{E}^{(d)}(\mc{M})$. Since 3 is a unit in $\Z_{\ell}$, multiplication by 3 is invertible in $\wh{E}^{(d)}(\mc{M})$.
Therefore
\begin{align}\label{E13}
E_{1}^{(d)}(\Q_{\ell})/3E_{1}^{(d)}(\Q_{\ell}) = 0.
\end{align}
Recall that for primes of multiplicative reduction, we have 
\begin{align*}
\#\wt{E}_{ns}(\F_{\ell}) =
\begin{cases}
\ell - 1 & \text{ if } \ell \text{ is split multiplicative}\\
\ell + 1 & \text{ if } \ell \text{ is nonsplit multiplicative}.
\end{cases}
\end{align*}
Furthermore, suppose $\ell$ is a prime of nonsplit multiplicative reduction for an elliptic curve $E/\Q$. By \cite[pp. 366, 378]{silverman2}, we have
\begin{align}\label{nonsplit}
\# E(\Q_{\ell})/E_{0}(\Q_{\ell}) = \begin{cases}1 & \text{if $\ord_{\ell}(\Delta_{E})$ is odd,}\\ 2 & \text{if $\ord_{\ell}(\Delta_{E})$ is even.}\end{cases}
\end{align}
We first consider the case when $\ell = 2$. This case is resolved with the following two lemmas.
\begin{lemma}\label{2unram1}
$E^{(d)}(\Q_{2})/3E^{(d)}(\Q_{2}) = \Z/3\Z$.
\end{lemma}
\begin{proof}
By our choice of $d$ in \eqref{dtab}, 2 is always of nonsplit multiplicative reduction for $E^{(d)}$.
Since $d \equiv 1 \imod{4}$, $\ord_{2}(\Delta_{E^{(d)}}) = \ord_{2}(\Delta_{E})$. Applying \eqref{nonsplit} to
the elliptic curve $E^{(d)}$ yields that $E^{(d)}(\Q_{2})/E_{0}^{(d)}(\Q_{2}) = 0$
if $\ord_{2}(\Delta_{E})$ is odd and is $\Z/2\Z$ if $\ord_{2}(\Delta_{E})$ is even.
We have the following commutative diagram with rows which are short exact sequences
\begin{align*}
\xymatrix{
0 \ar[r] &E_{0}^{(d)}(\Q_{2}) \ar[r]\ar[d]^{f = \times 3} &E^{(d)}(\Q_{2}) \ar[r]\ar[d]^{g = \times 3} &E^{(d)}(\Q_{2})/E_{0}^{(d)}(\Q_{2}) \ar[r]\ar[d]^{h = \times 3} &0\\
0 \ar[r] &E_{0}^{(d)}(\Q_{2}) \ar[r] &E^{(d)}(\Q_{2}) \ar[r] &E^{(d)}(\Q_{2})/E_{0}^{(d)}(\Q_{2})\ar[r] &0
}
\end{align*}
where $f$, $g$, and $h$ are the multiplication-by-3 map. Applying the Snake Lemma to the above diagram yields an exact sequence
containing
\begin{align*}
\xymatrix{
\cdots \ar[r] & \ker h \ar[r] & \cok f \ar[r] & \cok g \ar[r] & \cok h \ar[r] & 0.
}
\end{align*}
As $\ker h = 0$, $\cok f = E_{0}^{(d)}(\Q_{2})/3E_{0}^{(d)}(\Q_{2})$, $\cok g = E^{(d)}(\Q_{2})/3E^{(d)}(\Q_{2})$ and $\cok h = 0$,
\begin{align}\label{isom}
E_{0}^{(d)}(\Q_{2})/3E_{0}^{(d)}(\Q_{2}) \cong E^{(d)}(\Q_{2})/3E^{(d)}(\Q_{2}).
\end{align}
Now consider the following commutative diagram.
\begin{align*}
\xymatrix{
0 \ar[r] &E_{1}^{(d)}(\Q_{2}) \ar[r]\ar[d]^{f' = \times 3} &E_{0}^{(d)}(\Q_{2}) \ar[r]\ar[d]^{g' = \times 3} &E_{0}^{(d)}(\Q_{2})/E_{1}^{(d)}(\Q_{2}) \ar[r]\ar[d]^{h' = \times 3} &0\\
0 \ar[r] &E_{1}^{(d)}(\Q_{2}) \ar[r] &E_{0}^{(d)}(\Q_{2}) \ar[r] &E_{0}^{(d)}(\Q_{2})/E_{1}^{(d)}(\Q_{2})\ar[r] &0
}
\end{align*}
Since 2 is of nonsplit multiplicative reduction for $E^{(d)}$, it follows that $\#\wt{E}^{(d)}_{ns}(\F_{2}) = 3$ and hence $E^{(d)}_{0}(\Q_{2})/E^{(d)}_{1}(\Q_{2}) = \Z/3\Z$.
This implies that $\cok h' = \Z/3\Z$. From the Snake Lemma, we have an exact sequence containing the terms
\begin{align*}
\xymatrix{
\cdots \ar[r] & \cok f' \ar[r] & \cok g' \ar[r] & \cok h' \ar[r] & 0.
}
\end{align*}
As $\cok f' = E_{1}^{(d)}(\Q_{2})/3E_{1}^{(d)}(\Q_{2}) = 0$ by \eqref{E13} and $\cok g' = E_{0}^{(d)}(\Q_{2})/3E_{0}^{(d)}(\Q_{2})$, it follows that
$E_{0}^{(d)}(\Q_{2})/3E_{0}^{(d)}(\Q_{2}) = \Z/3\Z$. By \eqref{isom}, we have $E^{(d)}(\Q_{2})/3E^{(d)}(\Q_{2}) = \Z/3\Z.$
This completes the proof of Lemma \ref{2unram1}.
\end{proof}

\begin{lemma}\label{2unram}
The classes corresponding to $E^{(d)}(\Q_{2})/3E^{(d)}(\Q_{2})$ in $\sel_{3}(E^{(d)}/\Q)$ are unramified.
\end{lemma}
\begin{proof}
Fix a $P \in E^{(d)}(\Q_{2})$. Then there exists a $Q \in E^{(d)}(\ov{\Q}_{2})$ such that $3Q = P$. We want to show that the class
$\xi_{\sigma} = \{Q^{\sigma} - Q\}$ with $\sigma \in \gal(\ov{\Q}_{2}/\Q_{2})$ corresponding to $P$ in $\sel_{3}(E^{(d)}/\Q)$
is unramified, that is, $\xi_{\sigma} = 0$ when restricted to $H^{1}(I_{2}, E^{(d)}[3])$ where $I_{2} \subset \gal(\ov{\Q}_{2}/\Q_{2})$
is the inertia group at 2. For $\sigma \in I_{2}$, since an element of inertia acts trivially on $\wt{E}_{ns}^{(d)}(\F_{2})$, we have
$\wt{Q^{\sigma} - Q} = (\wt{Q})^{\sigma} - \wt{Q} = \wt{\mc{O}}.$
This implies that $Q^{\sigma} - Q \in E^{(d)}_{1}(\ov{\Q}_{2})$ which is a pro-2 group. Since we also have $Q^{\sigma} - Q \in E^{(d)}[3]$,
we must have $Q^{\sigma} - Q = 0$ since a pro-2 group cannot contain a nontrivial element with 3-torsion.
This completes the proof of Lemma \ref{2unram}.
\end{proof}

We recall the following group theoretic result.
\begin{lemma}\label{gplem}
Let $G$ be an abelian group and $H$ a subgroup of $G$ with $H/nH = 0$ for some $n$. If $(n, [G:H]) = 1$,
then $G/nG = 0$.
\end{lemma}
\begin{proof}
We write the group multiplicatively. Since $(n, [G:H]) = 1$, there exist integers $r$ and $s$
such that $rn + s[G:H] = 1$. Fix a $g \in G$. Since $g = g_{1}h$ for some representative $g_{1}$ of $G/H$
and $h \in H$, $g^{[G:H]} = g_{1}^{[G:H]}h^{[G:H]} = h^{[G:H]} \in H$ where the last equality is by
Lagrange's Theorem. Since every element of $H$ is an $n$th power, in particular so is $g^{[G:H]}$, and hence
$g = g^{rn + s[G:H]}$ is also an $n$th power. Writing the group additively, this implies that $G/nG = 0$.
This completes the proof of Lemma \ref{gplem}.
\end{proof}

Now suppose $\ell \mid N_{E}$ and $\ell \neq 2$.
By Lemma \ref{gplem} and \eqref{E13}, if we can show that $[E^{(d)}(\Q_{\ell}) : E_{1}^{(d)}(\Q_{\ell})]$
is prime to 3, then $E^{(d)}(\Q_{\ell})/3E^{(d)}(\Q_{\ell}) = 0$ and hence $\sel_{3}(E^{(d)}/\Q)$ is unramified
at all primes $\ell \mid N_{E}$, $\ell \neq 2$. We write
\begin{align*}
[E^{(d)}(\Q_{\ell}) : E_{1}^{(d)}(\Q_{\ell})] = [E^{(d)}(\Q_{\ell}) : E_{0}^{(d)}(\Q_{\ell})][E_{0}^{(d)}(\Q_{\ell}): E_{1}^{(d)}(\Q_{\ell})].
\end{align*}
To show that the left hand side is prime to 3, we show that both terms on the right hand side are prime to 3.

We first consider the case when $\ell$ is split for $E$. We have three subcases.

Suppose $\#\wt{E}_{ns}(\F_{\ell}) \equiv 0 \imod{3}$.
Then $\ell - 1 = \#\wt{E}_{ns}(\F_{\ell}) \equiv 0 \imod{3}$ and hence
$\ell \equiv 1 \imod{3}$. By our choice of $d$ in \eqref{dtab}, $d$ is not a square mod $\ell$ and
hence is not a square in $\Q_{\ell}$. Therefore by Proposition \ref{prop1}, $E^{(d)}/\Q_{\ell}$
has $\ell$ of nonsplit multiplicative reduction and hence $\#\wt{E}^{(d)}_{ns}(\F_{\ell}) = \ell + 1 \equiv 2 \imod{3}$.
As we have the exact sequence
\begin{align*}
\xymatrix{
0 \ar[r] & E_{1}^{(d)}(\Q_{\ell}) \ar[r] & E_{0}^{(d)}(\Q_{\ell}) \ar[r] & \wt{E}_{ns}^{(d)}(\F_{\ell}) \ar[r] & 0
}
\end{align*}
we have $[E_{0}^{(d)}(\Q_{\ell}) : E_{1}^{(d)}(\Q_{\ell})] = \#\wt{E}_{ns}^{(d)}(\F_{\ell}) \equiv 2 \imod{3}.$
As $\ell$ is of nonsplit multiplicative reduction for $E^{(d)}$, by \eqref{nonsplit}, we have $[E^{(d)}(\Q_{\ell}) : E_{0}^{(d)}(\Q_{\ell})] = 1$ or
$2$ depending on the parity of $\ord_{\ell}(\Delta_{E^{(d)}}) = \ord_{\ell}(\Delta_{E})$. Then it follows that
$[E^{(d)}(\Q_{\ell}) : E_{1}^{(d)}(\Q_{\ell})]$ is prime to 3 in this subcase.

Next, suppose $\#\wt{E}_{ns}(\F_{\ell}) \equiv 1 \imod{3}$.
Then $\ell -1 = \#\wt{E}_{ns}(\F_{\ell}) \equiv 1 \imod{3}$ and hence $\ell \equiv 2 \imod{3}$. By our choice of $d$ in \eqref{dtab},
$d$ is a square mod $\ell$ and hence is a square in $\Q_{\ell}$. Therefore by Proposition \ref{prop1}, $E^{(d)}/\Q_{\ell}$
has $\ell$ of split multiplicative reduction and hence
$[E_{0}^{(d)}(\Q_{\ell}) : E_{1}^{(d)}(\Q_{\ell})] = \#\wt{E}^{(d)}_{ns}(\F_{\ell}) = \ell - 1 \equiv 1 \imod{3}.$
As
$[E^{(d)}(\Q_{\ell}) : E_{0}^{(d)}(\Q_{\ell})] = \ord_{\ell}(\Delta_{E^{(d)}}) = \ord_{\ell}(\Delta_{E}),$
which is relatively prime to 3
by Assumption \ref{asmp2}, it follows that $[E^{(d)}(\Q_{\ell}) : E_{1}^{(d)}(\Q_{\ell})]$ is relatively prime to 3 in this subcase.

Finally, suppose $\#\wt{E}_{ns}(\F_{\ell}) \equiv 2 \imod{3}$.
Then $\ell - 1 = \#\wt{E}_{ns}(\F_{\ell}) \equiv 2\imod{3}$ which implies that $\ell \equiv 0 \imod{3}$ which is impossible since $\ell$ is prime.
Therefore this subcase can't happen.

Now consider the case when $\ell$ is nonsplit for $E$. We again have three subcases.

Suppose $\#\wt{E}_{ns}(\F_{\ell}) \equiv 0 \imod{3}$. Since $\ell$ is nonsplit,
$\ell + 1 = \#\wt{E}_{ns}(\F_{\ell}) \equiv 0 \imod{3}$ and hence $\ell \equiv 2 \imod{3}$.
By our choice of $d$ in \eqref{dtab}, $d$ is not a square mod $\ell$ and hence is not a square in $\Q_{\ell}$.
Therefore by Proposition \ref{prop1}, $E^{(d)}/\Q_{\ell}$ has $\ell$ of split multiplicative reduction
and hence $[E_{0}^{(d)}(\Q_{\ell}) : E_{1}^{(d)}(\Q_{\ell})] = \#\wt{E}^{(d)}_{ns}(\F_{\ell}) = \ell - 1 \equiv 1 \imod{3}.$
Since $[E^{(d)}(\Q_{\ell}): E_{0}^{(d)}(\Q_{\ell})] = \ord_{\ell}(\Delta_{E^{(d)}}) = \ord_{\ell}(\Delta_{E})$ is relatively
prime to 3, it follows that $[E^{(d)}(\Q_{\ell}) : E_{1}^{(d)}(\Q_{\ell})]$ is relatively prime to 3 in this subcase.

Suppose $\#\wt{E}_{ns}(\F_{\ell}) \equiv 1 \imod{3}$. Since $\ell$ is nonsplit,
$\ell + 1 = \#\wt{E}_{ns}(\F_{\ell}) \equiv 1 \imod{3}$ and hence $\ell \equiv 0 \imod{3}$.
This is impossible since $\ell$ is prime. Therefore this subcase can't happen.

Finally, suppose $\#\wt{E}_{ns}(\F_{\ell}) \equiv 2 \imod{3}$. Since $\ell$ is nonsplit,
$\ell + 1 = \#\wt{E}_{ns}(\F_{\ell}) \equiv 2 \imod{3}$ and hence $\ell \equiv 1 \imod{3}$.
By our choice of $d$ in \eqref{dtab}, $d$ is a square mod $\ell$ and hence is a square in $\Q_{\ell}$.
Therefore by Proposition \ref{prop1}, $E^{(d)}/\Q_{\ell}$ has $\ell$ of nonsplit multiplicative reduction
and hence $[E_{0}^{(d)}(\Q_{\ell}) : E_{1}^{(d)}(\Q_{\ell})] = \#\wt{E}^{(d)}_{ns}(\F_{\ell}) = \ell + 1 \equiv 2 \imod{3}.$
Since by \eqref{nonsplit} we have  $[E^{(d)}(\Q_{\ell}) : E_{0}^{(d)}(\Q_{\ell})]$ relatively prime to 3, it follows
that $[E^{(d)}(\Q_{\ell}) : E_{1}^{(d)}(\Q_{\ell})]$ is relatively prime to 3 in this subcase.

Therefore for $\ell \mid N_{E}$ with $\ell \neq 2$, we have shown $[E^{(d)}(\Q_{\ell}) : E_{1}^{(d)}(\Q_{\ell})]$ is relatively prime to 3 in all cases which
by Lemma \ref{gplem} and \eqref{E13} implies that
$E^{(d)}(\Q_{\ell})/3E^{(d)}(\Q_{\ell}) = 0$. This implies that $\sel_{3}(E^{(d)}/\Q)$ is
unramified at all places $\ell \mid N_{E}$.

Thus, for all squarefree $d$ satisfying the system of congruences specified by our choice of $d$ in $\eqref{dtab}$ (the system is finite since
we only have congruences for each $\ell \mid N_{E}$), $\sel_{3}(E^{(d)}/\Q)$ is unramified at all places away from 3.
That is, $\sel_{3}(E^{(d)}/\Q) \subset H^{1}(\Q, E^{(d)}[3]; \{3\}).$
This completes the proof of Proposition \ref{prop3}.
\end{proof}

\subsection{Decomposing the 3-Selmer Group}\label{3sel}
To further study the Selmer group, we make an assumption on $E/\Q$.
\begin{asmp}\label{asmp3}
Assume that:
\begin{enumerate}[$(i)$]
\item $E/\Q$ is an elliptic curve such that $E = E'/\langle P \rangle$ where $E'/\Q$ is an elliptic curve
with a rational 3-torsion point $P$.
\end{enumerate}
\end{asmp}

The above assumption yields a degree 3 isogeny $\phi: E' \rightarrow E$ given by the mod $P$ map.
Then there exists an $\F_{3}$ basis $\{P, Q\}$ of $E'[3]$ such that the action of $\sigma \in \gal(\ov{\Q}/\Q)$ on $E'[3]$ is of the form
$\smat{1}{\alpha}{0}{\om(\sigma)}$ where $\om$ is the mod 3 cyclotomic character. As $\ker\phi = \langle P\rangle$,
$\phi(Q) \neq 0$. Let $P' \in E[3]$ be such that $\{\phi(Q), P'\}$ is an $\F_{3}$ basis of $E[3]$.
Then the action of $\sigma \in \gal(\ov{\Q}/\Q)$ on $E[3]$ is of the form $\smat{\om(\sigma)}{\ast}{0}{1}$ since
for $\sigma \in \gal(\ov{\Q}/\Q)$, $\sigma(\phi(Q)) = \alpha\, \phi(P) + \om(\sigma)\phi(Q) = \om(\sigma)\phi(Q)$
in $E$. Since the Galois group $\gal(\ov{\Q}/\Q)$ acts on $E^{(d)}/\Q$ by the quadratic character $\chi_{d}$,
the matrix representing the action of $\sigma$ on $E^{(d)}[3]$ in the basis $\{\phi(Q)_{d}, P'_{d}\}$ (where $\phi(Q)_{d}$ and $P'_{d}$
are the points corresponding to $\phi(Q)$ and $P'$ in the quadratic twist) is
$\smat{\chi_{d}(\sigma)\om(\sigma)}{\ast}{0}{\chi_{d}(\sigma)}$.
Thus we have the short exact sequence of Galois modules
\begin{align}\label{cohomseq}
\xymatrix{0 \ar[r] & \F_{3}(\chi_{d}\om) \ar[r] & E^{(d)}[3] \ar[r] & \F_{3}(\chi_{d}) \ar[r] & 0}
\end{align}
where $\F_{3}(\chi_{d})$ and $\F_{3}(\chi_{d}\om)$ denote the module $\F_{3}$ endowed with the Galois actions
through the characters $\chi_{d}$ and $\chi_{d}\om$, respectively.
Taking cohomologies, we get a long exact sequence which include the terms
\begin{align*}
\xymatrix{\cdots \ar[r] &H^{1}(\Q, \F_{3}(\chi_{d}\om)) \ar[r] &H^{1}(\Q, E^{(d)}[3]) \ar[r] &H^{1}(\Q, \F_{3}(\chi_{d})) \ar[r] & \cdots.}
\end{align*}
Similarly, denoting $P_{d}$ and $Q_{d}$ the points corresponding to $P$ and $Q$ of the basis for $E'[3]$ in the quadratic twist
$E'^{(d)}$, we have that the matrix representing the action by $\sigma$ on $E'^{(d)}[3]$ is of the form $\smat{\chi_{d}(\sigma)}{\ast}{0}{\chi_{d}(\sigma)\om(\sigma)}.$
This yields the short exact sequence of Galois modules similar to \eqref{cohomseq},
\begin{align*}
\xymatrix{0 \ar[r] & \F_{3}(\chi_{d}) \ar[r] & E'^{(d)}[3] \ar[r] & \F_{3}(\chi_{d}\om) \ar[r] & 0}
\end{align*}
which gives the long exact sequence
\begin{align}\label{cohome1}
\xymatrix{\cdots \ar[r] &H^{1}(\Q, \F_{3}(\chi_{d})) \ar[r] &H^{1}(\Q, E'^{(d)}[3]) \ar[r] &H^{1}(\Q, \F_{3}(\chi_{d}\om)) \ar[r] & \cdots.}
\end{align}
By Proposition \ref{prop3}, assuming Assumptions \ref{asmp1} and \ref{asmp2}, $\sel_{3}(E^{(d)}/\Q) \subset H^{1}(\Q, E^{(d)}[3]; \{3\})$.
Let $$\phi: \sel_{3}(E^{(d)}/\Q) \rightarrow H^{1}(\Q, \F_{3}(\chi_{d}))$$ be the restriction of the map
$H^{1}(\Q, E^{(d)}[3]) \rightarrow H^{1}(\Q, \F_{3}(\chi_{d}))$ to the 3-Selmer group. Then to study the 3-Selmer group,
it suffices to study $\ker\phi$ and $\im\phi$.
\begin{asmp}\label{asmp4a}
Assume that:
\begin{enumerate}[$(ii)$]
\item $d \equiv 1 \imod{3}$.
\end{enumerate}
\end{asmp}
\begin{prop}\label{prop4}
Assume Assumptions \ref{asmp1}--\ref{asmp4a}.
Then $\ker\phi \subset H^{1}(\Q, \F_{3}(\chi_{d}\om); \{3\})$ and $\im\phi \subset H^{1}(\Q, \F_{3}(\chi_{d}); \varnothing)$.
\end{prop}
\begin{proof}
Applying Proposition \ref{prop3} yields that $\sel_{3}(E^{(d)}/\Q) \subset H^{1}(\Q, E^{(d)}[3]; \{3\})$. From \eqref{cohomseq} we have
$\ker\phi \subset H^{1}(\Q, \F_{3}(\chi_{d}\om); \{3\})$ and $\im\phi \subset H^{1}(\Q, \F_{3}(\chi_{d}); \{3\})$. Thus it remains
to show that $\im\phi$ is unramified at the prime 3.

Let $$\psi: \sel_{3}(E'^{(d)}/\Q) \rightarrow H^{1}(\Q, \F_{3}(\chi_{d}\om))$$
be the restriction of the map $H^{1}(\Q, E'^{(d)}[3]) \rightarrow H^{1}(\Q, \F_{3}(\chi_{d}\om))$ to $\sel_{3}(E'^{(d)}/\Q)$.
We reduce the problem to considering $E'/\Q$. The mod $\langle \phi(Q)\rangle$ map gives
$E/\gen{\phi(Q)} \cong E'$ and similarly the mod $\gen{\phi(Q)_{d}}$ map
gives $E^{(d)}/\gen{\phi(Q)_{d}} \cong E'^{(d)}$. Note that $\F_{3}(\chi_{d})$ factors through the mod $\gen{\phi(Q)_{d}}$ map,
that is,
\begin{align*}
\xymatrix{
E^{(d)}(\Q_{3})[3] \ar[d] \ar[dr]^{\slash\gen{\phi(Q)_{d}}} &\\
\F_{3}(\chi_{d}) \ar[r] & E'^{(d)}(\Q_{3})[3]
}
\end{align*}
Thus to show that $\im\phi$ is unramified at the prime 3, it suffices to show that $\ker\psi$ is unramified at the prime 3 (since
we have the exact sequence \eqref{cohome1}).

From \eqref{cohome1}, we know that $\ker\psi$ is equal to the intersection of $\sel_{3}(E'^{(d)}/\Q)$ and the image
of $H^{1}(\Q, \F_{3}(\chi_{d}))$. From the definition of $\sel_{3}(E'^{(d)}/\Q)$, any element $\xi$ when restricted to $\Q_{3}$
is of the form $\xi(\sigma) = \mc{Q}^{\sigma} - \mc{Q}$ for all $\sigma \in \gal(\ov{\Q}/\Q)$ where $\mc{Q} \in [3]^{-1}E'^{(d)}(\Q_{3})$
is some point in $E'^{(d)}(\ov{\Q}_{3})$ that becomes a $\Q_{3}$ rational point under the multiplication by 3 map. Note that 3 is a prime
of good reduction for $E'^{(d)}$ (since isogenous elliptic curves have the same conductor) which implies that $\mc{Q}$ which is defined
over a finite extension of $\Q_{3}$ is a point of good reduction.

If $\xi$ is also in the image of $H^{1}(\Q, \F_{3}(\chi_{d}))$,
without loss of generality, then $\xi(\sigma) = \mc{Q}^{\sigma} - \mc{Q} \in \F_{3}(\chi_{d})$ for all $\sigma$. Consider $\sigma \in I_{3}$,
the inertia subgroup in $\gal(\ov{\Q}_{3}/\Q_{3})$. By commutativity of the reduction map and the Galois action,
$\wt{\mc{Q}^{\sigma} - \mc{Q}} = (\wt{\mc{Q}})^{\sigma} - \wt{\mc{Q}} = \wt{\mc{O}}$ where the last equality is because $\sigma \in I_{3}$.
Since $\mc{Q}^{\sigma} - \mc{Q} \in \{0, P_{d}, -P_{d}\} = \F_{3}(\chi_{d})$, as long as $P_{d}$ does not reduce to the point at infinity,
we must have $\xi(\sigma) = \mc{Q}^{\sigma} - \mc{Q} = \mc{O}$. Since by Assumption \ref{asmp4a},
we assumed $d\equiv 1 \imod{3}$, $\Q_{3}(\sqrt{d}) = \Q_{3}$ which implies that
$E'(\Q_{3}) \cong E'^{(d)}(\Q_{3})$. Thus $\wt{P}_{d} = \wt{P}$.
Recall that $P$ is a rational 3-torsion point for $E'$. Then by the analogue of the Nagell-Lutz Theorem for
local fields \cite[Theorem VII.3.4]{silverman1} we have $\wt{P} \neq \wt{\mc{O}}$.
Thus $\ker\psi$ is unramified at the prime 3 and hence so is $\im\phi$.
This completes the proof of Proposition \ref{prop4}.
\end{proof}

\subsection{Determining the 3-Selmer Group}\label{furthercond}
In this section, for notational convenience, let $G_{K} := \gal(\ov{K}/K)$ and $G_{L/K} := \gal(L/K)$.
For a number field $K$, let $h(K)$ denote the class number of $K$.
We make the following assumption on $d$:
\begin{asmp}\label{asmp4}
Assume that:
\begin{enumerate}[$(ii)$]
\item $3 \nmid h(\Q(\sqrt{d}))$.
\end{enumerate}
\end{asmp}
Let $K := \Q(\sqrt{d})$ be the quadratic extension associated with the quadratic twist by $\chi_{d}$. The Galois group $G_{K}$
lies in the kernel of the quadratic character $\chi_{d}$ and hence $G_{K}$ acts trivially on $\F_{3}(\chi_{d})$. Using the Inflation-Restriction sequence
(\cite[Proposition 2, p. 105]{washington}, with $G = G_{\Q}$ and $H = G_{\ov{\Q}/K}$) yields
\begin{align*}
\xymatrix{
0 \ar[r] & H^{1}(G_{K/\Q}, \F_{3}(\chi_{d})^{G_{\ov{\Q}/K}}) \ar[r] & H^{1}(\Q, \F_{3}(\chi_{d})) \ar[r] & H^{1}(G_{\ov{\Q}/K}, \F_{3}(\chi_{d}))^{G_{K/\Q}}
}.
\end{align*}
Since $\gal(\ov{\Q}/K) = \gal(\ov{K}/K)$, this exact sequence becomes
\begin{align*}
\xymatrix{
0 \ar[r] & H^{1}(G_{K/\Q}, \F_{3}(\chi_{d})) \ar[r] & H^{1}(\Q, \F_{3}(\chi_{d})) \ar[r] & H^{1}(G_{K}, \F_{3}(\chi_{d}))^{G_{K/\Q}}
}.
\end{align*}
Taking the everywhere unramified subgroup, we have
\begin{align}\label{fcondeq1}
\xymatrix{
0 \ar[r] & H^{1}(G_{K/\Q}, \F_{3}(\chi_{d}); \varnothing) \ar[r] & H^{1}(\Q, \F_{3}(\chi_{d}); \varnothing) \ar[r] & H^{1}(G_{K}, \F_{3}(\chi_{d}); \varnothing)^{G_{K/\Q}}
}.
\end{align}
Note that $\#\gal(K/\Q) = 2$ and $\#\F_{3}(\chi_{d}) = 3$.
Recall that an abelian group $A$ is uniquely divisible by $m$ if for each $a \in A$, there is a $b \in A$ such that
$a = mb$. Then we have the following proposition from group cohomology.
\begin{prop}[\cite{weiss}, Proposition 3-1-11, p. 90]\label{groupcohom}
If $G$ is a finite group of order $m$ and $A$ is a $G$-module which is uniquely divisible by $m$ then $H^{k}(G, A) = 0$ for all $k$.
\end{prop}
Since $(2, 3) = 1$, $\F_{3}(\chi_{d})$ is uniquely divisible by 2 and hence by the above proposition we have $H^{1}(G_{K/\Q}, \F_{3}(\chi_{d}); \varnothing) = 0$.
Note that as $G_{K}$ acts trivially on $\F_{3}(\chi_{d})$, $H^{1}(G_{K}, \F_{3}(\chi_{d})) = \hom(G_{K}, \F_{3}(\chi_{d}))$ is
in one-to-one correspondence with cubic extensions of $K$ (by the same argument as on the top of \cite[p. 104]{washington}).
The everywhere unramified condition in $H^{1}(G_{K}, \F_{3}(\chi_{d}); \varnothing)$ translates to the cubic extensions
being everywhere unramified. As $3 \nmid h(\Q(\sqrt{d}))$ by Assumption \ref{asmp4}, from class field theory,
the only such extension is the trivial extension and hence it follows that $H^{1}(G_{K}, \F_{3}(\chi_{d}); \varnothing) = 0$.
Therefore \eqref{fcondeq1} implies that
\begin{align}\label{triangle}
H^{1}(\Q, \F_{3}(\chi_{d}); \varnothing) = 0.
\end{align}

We now compute $H^{1}(\Q, \F_{3}(\chi_{d}\om); \{3\})$. Let $G_{p} = G_{\Q_{p}}$ and $I_{p} \subset G_{p}$
the inertia subgroup. We recall an identity from \cite{ddt}. Let $M$ be a continuous discrete $G_{\Q}$-module of finite cardinality
with $M^{\ast} = \hom(M, \mu_{n}(\ov{\Q}))$ where $n$ is such that $nM = 0$ and $\mu_{n}(\ov{\Q})$ denotes
the group of $n$th roots of unity in $\ov{\Q}$. By a collection of local conditions for $M$, we mean a collection
$\mc{L} = \{L_{v}\}$ of subgroups $L_{v} \subset H^{1}(G_{v}, M)$ as $v$ runs through the primes of $\Q$ with
$L_{v} = H^{1}(G_{v}/I_{v}, M^{I_{v}})$ for all but finitely many $v$. By local Tate duality, $\mc{L}^{\ast} = \{L_{v}^{\perp}\}$
is a collection of local conditions for $M^{\ast}$. If $\mc{L}$ is a collection of local conditions for $M$, define
the corresponding Selmer group $H^{1}_{\mc{L}}(\Q, M)$ to be the subgroup $x \in H^{1}(\Q, M)$ such that
for all places $v$ of $\Q$, we have $\res_{v}(x) \in L_{v} \subset H^{1}(G_{v}, M)$.

\begin{prop}[\cite{ddt}, Theorem 2.18]
If $\mc{L}$ is a collection of local conditions for $M$, then $H^{1}_{\mc{L}}(\Q, M)$ is finite and
\begin{align}\label{ddtprop}
\frac{\# H^{1}_{\mc{L}}(\Q, M)}{\# H^{1}_{\mc{L}^{\ast}}(\Q, M^{\ast})} = \frac{\# H^{0}(\Q, M)}{\# H^{0}(\Q, M^{\ast})}\prod_{v \leq \infty}\frac{\# L_{v}}{\# H^{0}(\Q, M)}.
\end{align}
\end{prop}

In our case, $M = \F_{3}(\chi_{d})$ and $M^{\ast} = \F_{3}(\chi_{d}\om)$. For finite places $v$, we have $L_{v} = H^{1}(\F_{v}, \F_{3}(\chi_{d})^{I_{v}})$
(note that $G_{p}/I_{p} \cong \gal(\ov{\F}_{p}/\F_{p})$) and for the infinite place, $L_{\infty} = 0$. The dual condition
$\mc{L}^{\ast}$ is $L_{v}^{\perp} = H^{1}(\F_{v}, \F_{3}(\chi_{d}\om)^{I_{v}})$ at all places away from 3 and places no restriction
at the prime 3. Hence \eqref{ddtprop} yields that
\begin{align*}
\frac{\# H^{1}(\Q, \F_{3}(\chi_{d}); \varnothing)}{\# H^{1}(\Q, \F_{3}(\chi_{d}\om); \{3\})} = \frac{\# H^{0}(\Q, \F_{3}(\chi_{d}))}{\# H^{0}(\Q, \F_{3}(\chi_{d}\om))}\cdot \frac{1}{\# H^{0}(\R, \F_{3}(\chi_{d}))}\prod_{\ell < \infty}\frac{\# H^{1}(\F_{\ell}, \F_{3}(\chi_{d})^{I_{\ell}})}{\# H^{0}(\Q_{\ell}, \F_{3}(\chi_{d}))}.
\end{align*}
By Lemma 1 of \cite{washington}, $\# H^{1}(\F_{\ell}, \F_{3}(\chi_{d})^{I_{\ell}}) = \# H^{0}(\Q_{\ell}, \F_{3}(\chi_{d}))$ and
since $\chi_{d}$ and $\chi_{d}\om$ are nontrivial, $H^{0}(\Q, \F_{3}(\chi_{d})) = H^{0}(\Q, \F_{3}(\chi_{d}\om)) = 0$.
We also have $H^{0}(\R, \F_{3}(\chi_{d})) = \Z/3\Z$ if $d > 0$ and is trivial if $d < 0$.
Then
\begin{align*}
\frac{\# H^{1}(\Q, \F_{3}(\chi_{d}); \varnothing)}{\# H^{1}(\Q, \F_{3}(\chi_{d}\om); \{3\})} =
\begin{cases}
1/3 & \text{ if } d> 0\\
1 & \text{ if } d < 0
\end{cases}
\end{align*}
which implies that
\begin{align}\label{compeq1}
H^{1}(\Q, \F_{3}(\chi_{d}\om); \{3\}) =
\begin{cases}
\Z/3\Z & \text{ if } d > 0\\
0 & \text{ if } d< 0.
\end{cases}
\end{align}
%
%
%
%
%
%
%
%
To finally compute the 3-Selmer group, we make the following assumption pertaining to all $d$ chosen above.
\begin{asmp}\label{asmp5a}
Assume that:
\begin{enumerate}[$(ii)$]
\item $d > 0$ and for all $d$ chosen in Assumptions \ref{asmp1}, \ref{asmp2}, \ref{asmp4a}, and \ref{asmp4}, we have $\om(E^{(d)}) = -1$ and $E^{(d)}(\Q)_{\textrm{tors}} = 0$.
\end{enumerate}
\end{asmp}

\begin{prop}\label{3selrank}
If $d$ and $E/\Q$ satisfy Assumptions \ref{asmp1}--\ref{asmp5a}, then $\sel_{3}(E^{(d)}/\Q) = \Z/3\Z.$
\end{prop}
\begin{proof}
With $\phi: \sel_{3}(E^{(d)}/\Q) \rightarrow H^{1}(\Q, \F_{3}(\chi_{d}))$,
by Proposition \ref{prop4} and \eqref{triangle}, we have $\im\phi \subset H^{1}(\Q, \F_{3}(\chi_{d}); \varnothing) = 0$. Then
$\sel_{3}(E^{(d)}/\Q) \subset \ker\phi \subset H^{1}(\Q, \F_{3}(\chi_{d}\om); \{3\}) = \Z/3\Z$
as $d > 0$ and \eqref{compeq1}.
This implies that $\sel_{3}(E^{(d)}/\Q)$ is either $0$ or $\Z/3\Z$.

Suppose $\sel_{3}(E^{(d)}/\Q) = 0$. Since $E^{(d)}(\Q)_{\textrm{tors}} = 0$, by \cite[Proposition 5.10(c)]{sdistr},
\begin{align}\label{s3inf}
\sel_{3^{\infty}}(E^{(d)}/\Q)[3] \cong \sel_{3}(E^{(d)}/\Q) = 0.
\end{align}
Recall that
$\sel_{3^{\infty}}(E^{(d)}/\Q) \cong (\Q_{3}/\Z_{3})^{s} \oplus F$ for some finite abelian 3-group $F$.
If $s \geq 1$, then $(\Q_{3}/\Z_{3})^{s}[3] \cong (\Z/3\Z)^{s}$ and hence contradicts \eqref{s3inf}.
Therefore $s = 0$. If $F \neq 0$, then $F[3] \neq 0$, as $F$ is a finite 3-group and by Cauchy's Theorem contains an element
of order 3. This again contradicts \eqref{s3inf} and hence we must have $F = 0$.
Therefore $\sel_{3^{\infty}}(E^{(d)}/\Q) = 0$.
However, since $\om(E^{(d)}) = -1$, by the $p$-parity conjecture proven in \cite{dokchitser},
we must have $s \equiv 1 \imod{2}$. This is a contradiction.
Therefore $\sel_{3}(E^{(d)}/\Q) = \Z/3\Z.$
This completes the proof of Proposition \ref{3selrank}.
\end{proof}

\begin{rem}
The proof of the above proposition also allows us to determine the $3^{\infty}$-Selmer group.
Indeed, since $E^{(d)}(\Q)_{\textrm{tors}} = 0$, $\sel_{3^{\infty}}(E^{(d)}/\Q)[3] \cong \sel_{3}(E^{(d)}/\Q) = \Z/3\Z$.
As $(\Q_{3}/\Z_{3})[3] = \Z/3\Z$, it follows that we must have $F = 0$ and $s = 1$. Therefore
we have $\sel_{3^{\infty}}(E^{(d)}/\Q) = \Q_{3}/\Z_{3}.$
\end{rem}

\subsection{Proof of Theorem \ref{resthm} and Corollary \ref{rescor}}\label{respfs}
In Assumption \ref{asmp2}, we assumed that $\ord_{2}(\Delta_{E})$ is odd.
To make use of the Davenport-Heilbronn/Nakagawa-Horie estimates (\cite{nakagawahorie, taya}) regarding class groups, we need to make
the following assumption.
\begin{asmp}\label{asmp5}
Assume that:
\begin{enumerate}[$(i)$]
\item $\ord_{2}(N_{E}) = 1$.
\end{enumerate}
\end{asmp}
Fix a semistable elliptic curve $E/\Q$ satisfying Assumptions \ref{asmp1}--\ref{asmp5}. We show that a positive proportion
of its quadratic twists have 3-Selmer rank 1 and global root number $-1$. We recall from Assumptions \ref{asmp1}--\ref{asmp5}
that we choose $d$ to satisfy
\begin{equation}\label{dcond}
\begin{split}
&d > 0\\
&d \text{ is squarefree}\\
\end{split}
\quad\quad\quad
\begin{split}
d & \equiv 1 \imod{3}\\
d & \equiv 1 \imod{4}
\end{split}
\quad\quad\quad
\begin{split}
&(d, \Delta_{E}) = 1\\
&3 \nmid h(\Q(\sqrt{d}))
\end{split}
\end{equation}
and for each $\ell \mid N_{E}$ we require $d$ as follows
\begin{align}\label{dtab2}
\begin{array}{ccc|c}
\text{Hypotheses on $\ell$} &&& \text{Choice for $d$}\\\hline
\text{$E$ has split reduction at $\ell$} & \ell = 2 && d \equiv 5 \imod{8}\\
& \ell \neq 2& \ell \equiv 1 \imod{3} & d \equiv n_{1} \imod{\ell}\\
& \ell \neq 2& \ell \equiv 2 \imod{3} & d \equiv 1 \imod{\ell}\\
\text{$E$ has nonsplit reduction at $\ell$} & \ell = 2 && d \equiv 1 \imod{8}\\
& \ell \neq 2& \ell \equiv 1 \imod{3} & d \equiv 1 \imod{\ell}\\
& \ell \neq 2& \ell \equiv 2 \imod{3} & d \equiv n_{2} \imod{\ell}
\end{array}
\end{align}
where $n_{1}$ and $n_{2}$ are chosen so that the Legendre symbol $(n_{1}/\ell) = -1$ and $(n_{2}/\ell) = -1$.
Furthermore, by Assumption \ref{asmp5a}, for whatever $d$ we choose above, we assume that it automatically is such that $\om(E^{(d)}) = -1$ and $E^{(d)}(\Q)_{tors} = 0$.
\begin{rem}
Note that the $\ell = 2$, split choice is slightly different from \eqref{dtab}.
In particular, we chose $d \equiv 5 \imod{8}$ not just $d \not\equiv 1 \imod{8}$. This is a convenient choice
as to absorb the $d \equiv 1 \imod{4}$ condition later. We also made a choice for the $\ell \neq 2$, $\ell \equiv 2 \imod{3}$, split case
and the $\ell \neq 2$, $\ell \equiv 1 \imod{3}$, nonsplit case. \hfill\qed
\end{rem}

We would like to show that
\begin{align}\label{countd}
\lim_{x \rightarrow \infty}\frac{\#\{d : 0 < d < x, d \text{ satisfies \eqref{dcond} and \eqref{dtab2}}\}}{x} > 0.
\end{align}
By our choice of $d$ in $\eqref{dtab2}$, we have $(d, N_{E}) = 1$ and hence the assumption that $(d, \Delta_{E}) = 1$ is redundant. Note that using
the Chinese Remainder Theorem in the choice of $d$ in \eqref{dtab2} above yields the congruence
$d \equiv \alpha \imod{4^{s(N_{E})}N_{E}}$
with $\alpha \neq 0$ where $s(n) = 1$ if $n$ is even and $0$ if $n$ is odd. Combining this with the assumption that
$d \equiv 1 \imod{3}$ yields a congruence mod $4^{s(N_{E})}\cdot 3N_{E}$ since 3 is a prime of good reduction
and hence relatively prime to $N_{E}$.

If $N_{E}$ is even, then for some $\beta \neq 0$, we have $d \equiv \beta \imod{12N_{E}}$ and $d \equiv 1 \imod{4}$. However,
as $N_{E}$ is even, 2 is a prime of multiplicative reduction and the condition that $d \equiv 1 \imod{4}$ is already included
in $d \equiv \beta \imod{12N_{E}}$, since we have chosen $d \equiv 1 \imod{8}$ or $d \equiv 5 \imod{8}$.

If $N_{E}$ is odd, then we have $d \equiv \beta \imod{3N_{E}}$ and $d \equiv 1 \imod{4}$.
In this case, the Chinese Remainder Theorem yields a congruence $\gamma\imod{12N_{E}}$ since $(3N_{E}, 4) = 1$.

Therefore regardless, we have $d \equiv \delta \imod{12N_{E}}$ for some $\delta$. Therefore we want to show that we have a positive proportion
of positive squarefree $d$ satisfying
\begin{align*}
d \equiv \delta \imod{12N_{E}} \quad\quad\quad 3 \nmid h(\Q(\sqrt{d}))
\end{align*}
(note that since $(d, N_{E}) = 1$, we have $(\delta, 12N_{E}) = 1$ as $d \equiv 1 \imod{3}$ and $d\equiv 1 \imod{4}$).

The Nakagawa-Horie estimates deal with quadratic fields whose discriminants
are $m \imod{N}$ and whose class number is not divisible by 3. For these estimates to hold
we need the following condition (see \cite[p. 21]{nakagawahorie})
\begin{quote}
If an odd prime number $p$ is a common divisor of $m$ and $N$, then $p^{2}$
divides $N$ but not $m$. Further if $N$ is even then $(i)$ 4 divides $N$ and
$m \equiv 1 \imod{4}$, or $(ii)$ $16$ divides $N$ and $m \equiv 8$ or $12 \imod{16}$.
\end{quote}
Take $m := \delta$ and $N := 12N_{E}$. As $(\delta, 12N_{E}) = 1$ and $\ord_{2}(N_{E}) = 1$ under Assumption \ref{asmp5},
the above assumptions on $m$ and $N$ hold.

Let $K^{+}(x)$ be the set of all real quadratic fields $k$ with discriminant $\Delta_{k} < x$ and let
$$K^{+}(x, m, N) = \{k \in K^{+}(x) : \Delta_{k} \equiv m \imod{N}\}.$$
As the congruence $\Delta_{k} \equiv m \imod{N}$ implies that $\Delta_{k} \equiv 1 \imod{4}$, it follows that
\begin{align*}
K^{+}(x, m, N) = \{\Q(\sqrt{d}) : 0 < d < x, \mu(d)^{2} = 1, d \equiv m \imod{N}\}.
\end{align*}
Let $$K_{\ast}^{+}(x, m, N) = \{k \in K^{+}(x, m, N) : h(k) \not\equiv 0 \imod{3}\}.$$
Note that $\# K_{\ast}^{+}(x, m, N)$ is precisely the set of $d$ that we want to count in \eqref{countd}.
For a quadratic field $k$, denote by $h_{3}^{\ast}(k)$ the number of ideal classes of $k$ whose
cubes are principal.
By \cite[pp. 1289-1290]{taya},
\begin{align*}
\lim_{x \rightarrow \infty}\frac{\# K_{\ast}^{+}(x, m, N)}{x} &\geq \lim_{x \rightarrow \infty}\frac{3}{2}\cdot \frac{\#K^{+}(x, m, N)}{x} - \frac{1}{2}\cdot \frac{\sum_{k \in K^{+}(x, m, N)}h_{3}^{\ast}(k)}{x}\\
&=\lim_{x \rightarrow \infty}\frac{5}{6}\cdot \frac{\# K^{+}(x, m, N)}{x} = \frac{5}{\pi^{2}\varphi(12N_{E})}\prod_{p \mid 12N_{E}}\frac{p}{p + 1} > 0
\end{align*}
which proves \eqref{countd}.
Then we have shown that for $E$ and $d$ satisfying Assumptions \ref{asmp1}--\ref{asmp5}, a positive proportion of the quadratic twists
of $E$ have $\sel_{3}(E^{(d)}/\Q) = \Z/3\Z$ and $\om(E^{(d)}) = -1$. This proves Theorem \ref{resthm}.

Corollary \ref{rescor} is immediate since $\om(E^{(d)}) = -1$
implies that $\rkan(E^{(d)}) \equiv 1 \imod{2}$. Under the assumption of the rank part of the Birch and Swinnerton-Dyer conjecture,
we have $\rk(E^{(d)}) \equiv 1 \imod{2}$. Since the 3-Selmer rank of $E^{(d)}$ is 1, we have $\rk(E^{(d)}) \leq 1$
and hence $\rk(E^{(d)}) = 1$.

Of course Corollary \ref{rescor} is true even if we knew only that $\sel_{3}(E^{(d)}/\Q) \subset \Z/3\Z$
since this still implies that the 3-Selmer rank is $\leq 1$ and hence $\rk(E^{(d)}) \leq 1$. However,
for completeness, we have determined the precise 3-Selmer group and have aimed for unconditional results
as opposed to ones that assume the Birch and Swinnerton-Dyer conjecture.

We now give infinitely many explicit elliptic curves $E$ which satisfy Assumptions \ref{asmp1}--\ref{asmp5}.

%% file: qt_example.tex
\section{Proof of Theorem \ref{exthm}}\label{example}
\subsection{28 Infinite Families}
We first construct 28 nonisomorphic infinite families of elliptic curves $E/\Q$ satisfying the following properties
\begin{enumerate}[(1)]
\item $E = C/\langle P \rangle$ where $C/\Q$ is an elliptic curve with rational 3-torsion point $P$,\label{pcond}
\item $E/\Q$ is semistable,\label{cond2}
\item 3 is a prime of good reduction, and\label{cond3}
\item At all places $v$ of bad reduction $3 \nmid \ord_{v}(\Delta_{E})$.\label{cond4}
\end{enumerate}
The infinite families of elliptic curves we construct here will serve as the main starting point for constructing infinitely
many elliptic curves each of which satisfy Theorem \ref{resthm}.

For $m = 1, 2, 5, 7, 8, 10, 11, 13, 14, 16, 17, 19, 22, 23$ and $n \in \Z_{\geq 0}$, consider
\begin{align}\label{nchoice}
\mp 2(m + 24n)(62208n^{2} + (5184m \mp 432)n + (108m^{2} \mp 18m + 1)) = \mp 2^{i}\delta_{m}(n)
\end{align}
where $i = 1$ for $m$ odd, $i = 2$ for $m \neq 8, 16$ and even, and $i = 4$ for $m = 8, 16$.
Let $n$ be such that $\delta_{m}(n)$ is squarefree (possibly containing a factor of 2 when $m = 8, 16$).

Recall our convention regarding signs in \eqref{nchoice} first mentioned in the footnote associated to the discussion
of \eqref{choice}. If we choose to consider $-2^{i}\delta_{m}(n)$ in \eqref{nchoice}, then
in all subsequent expressions, $\mp$ is to be read as ``$-$" and $\pm$ is to be read as ``$+$".
If instead we choose to consider $+2^{i}\delta_{m}(n)$ in \eqref{nchoice}, then in all subsequent expressions,
$\mp$ is to be read as ``$+$" and $\pm$ is to be read as ``$-$".

Define an elliptic curve $C_{m, n}/\Q$ by (not necessarily the global minimal equation)
$$Y^{2} = X^{3} + d(aX + b)^{2}$$ where $d = 1$,
\begin{align*}
a = 18(m + 24n)\mp 1, \quad \text{ and } \quad b = (4/27)(a^{3} \pm 1).
\end{align*}
%
By how $C_{m, n}$ is defined, it has a rational 3-torsion point $P := (0, b\sqrt{d}) = (0, b)$. From \cite[p. 372]{cohen}, an equation for
$C_{m, n}/\langle P \rangle$ is
\begin{align}\label{emnold}
Y^{2} = X^{3} + D(AX + B)^{2}
\end{align}
where $D = -3d = -3$, $A = a$, and $B = (27b - 4a^{3}d)/9 = \pm 4/9$.
Note that the above equation is not the global minimal Weierstrass equation for \eqref{emnold}. For reference,
we compute the discriminant of \eqref{emnold}. Let $\alpha := 18(m + 24n)$ and $\beta := 2(m + 24n)$. Then $\alpha = 3^{2}\beta$ and $A = \alpha \mp 1$.
The discriminant of \eqref{emnold} is
\begin{align}\label{oldisc}
16B^{3}D^{2}(4DA^{3} - 27B) = \mp 2^{12}3^{-3}(A^{3} \pm 1).
\end{align}
Since
$A^{3} \pm 1 = \alpha^{3} \mp 3\alpha^{2} + 3\alpha = 3^{3}\beta(3^{3}\beta^{2} \mp 3^{2}\beta + 1),$
it follows from \eqref{oldisc} that the discriminant is
\begin{align}\label{oldisc1}
16B^{3}D^{2}(4DA^{3} - 27B) = \mp 2^{12}\beta(3^{3}\beta^{2} \mp 3^{2}\beta + 1).
\end{align}
We now give the global minimal Weierstrass equation for \eqref{emnold}.
Let $r := 1/3 + A^{2}$. Applying the transformations
$X \mapsto 4x + r, Y \mapsto 8y + 4x$
to \eqref{emnold} shows that the global minimal Weierstrass equation is given by
\begin{align}\label{mineqn}
y^{2} + xy = x^{3} + H(m, n)x + J(m, n)
\end{align}
where
\begin{align*}
H(m, n) = \frac{1}{16}(2ABD + 2A^{2}Dr + 3r^{2}) \quad\text{ and }\quad J(m, n) = \frac{1}{64}(B^{2}D + 2ABDr + A^{2}Dr^{2} + r^{3}).
\end{align*}
We define \eqref{mineqn} to be the elliptic curve $E_{m, n}/\Q$. By construction, $E_{m, n}$ satisfies
Property \eqref{pcond}.
\begin{rem}\label{remtrans}
The values for $H(m, n)$ and $J(m, n)$ were found using the change of variables formula in \cite[p. 45]{silverman1} with $a_{1} = 0$, $a_{2} = DA^{2}$,
$a_{3} = 0$, $a_{4} = 2ABD$, $a_{6} = DB^{2}$ the coefficients corresponding to \eqref{emnold} and
$a_{1}' = 1$, $a_{2}' = 0$, $a_{3}' = 0$, $a_{4}' = H(m, n)$, $a_{6}' = J(m, n)$ the coefficients corresponding to \eqref{mineqn}
where $r = 1/3 + A^{2}$, $s = 1$, $t = 0$, and $u = 2$. We will denote the invariants corresponding to \eqref{mineqn} with a $\prime$
(such as $b_{1}'$, $b_{2}'$, etc.) and those corresponding to \eqref{emnold} without one (such as $b_{1}$, $b_{2}$, etc.).\hfill\qed
\end{rem}
It is not immediate that $H(m, n)$ and $J(m, n)$ are integers. We check that this is indeed true.
\begin{lemma}\label{exlem1}
$H(m, n) \in \Z$.
\end{lemma}
\begin{proof}
Applying the definitions of $A$, $B$, and $D$ yields that
\begin{align*}
H(m, n) = \frac{1}{16}(2ABD + 2A^{2}Dr + 3r^{2}) = \frac{1}{48}(1 \mp 8A - 9A^{4}).
\end{align*}
Since $A \equiv \mp 1 \imod{3}$,
$1 \mp 8A - 9A^{4} \equiv 1 \mp 8A \equiv 1 + 8 \equiv 0 \imod{3}$. As $A \equiv 2m \mp 1 \imod{16}$,
\begin{align}\label{exlem1eq1}
1 \mp 8A - 9A^{4} \equiv 1 \mp 8(2m \mp 1) - 9A^{4} \equiv 9(1 - A^{4}) \imod{16}.
\end{align}
Since $(2m \mp 1)^{4} = 16m^{4} \mp 32m^{3} + 24m^{2} \mp 8m + 1$,
$1 - A^{4} \equiv 8(-3m^{2} \pm m) \equiv 0 \imod{16}$ where the last congruence is because
$-3m^{2} \pm m \equiv m^{2} \pm m \equiv m(m + 1) \equiv 0 \imod{2}$. Combining this and \eqref{exlem1eq1}
yields that $1 \mp 8A - 9A^{4} \equiv 0 \imod{16}$. Thus $1 \mp 8A - 9A^{4} \equiv 0 \imod{48}$ and hence
$H(m, n) \in \Z$. This completes the proof of Lemma \ref{exlem1}.
\end{proof}
\begin{lemma}\label{exlem2}
$J(m, n) \in \Z$.
\end{lemma}
\begin{proof}
Computation yields that
\begin{align*}
J(m, n) = \frac{1}{64}(B^{2}D + 2ABDr + A^{2}Dr^{2} + r^{3}) = -\frac{1}{2^{6}3^{2}}(5 \pm 8A \pm 24A^{3} + 9A^{4} + 18A^{6}).
\end{align*}
Since $A \equiv \mp 1 \imod{9}$, $A^{3} \equiv \mp 1 \imod{9}$ and hence
$5 \pm 8A \pm 24A^{3} + 9A^{4} + 18A^{6} \equiv 5 - 8 - 24 \equiv 0 \imod{9}.$

We now claim that $5 \pm 8A \pm 24A^{3} + 9A^{4} + 18A^{6} \equiv 0 \imod{64}$. Recall that $\alpha = 18(m + 24n)$ and $A = \alpha \mp 1$.
Expanding $A^{3}$, $A^{4}$, and $A^{6}$ in terms of $\alpha$ yields that
\begin{align}\label{alphacalc}
5 \pm 8A \pm 24A^{3} + 9A^{4} + 18A^{6} = 18\alpha^{6} \mp 108\alpha^{5} + 279\alpha^{4} \mp 372\alpha^{3} + 252\alpha^{2} \mp 64\alpha.
\end{align}
Since $\alpha = 18(m + 24n)$, $\alpha^{6} \equiv 0 \imod{64}$ and $44\alpha^{5} \equiv 0 \imod{64}$. Then by the above centered equation,
writing $\alpha = 2\gamma$ yields that
$5 \pm 8A \pm 24A^{3} + 9A^{4} + 18A^{6} \equiv 23\alpha^{4} \mp 52\alpha^{3} + 60\alpha^{2} \imod{64} = 2^{4}(23\gamma^{4} \mp 26\gamma^{3} + 15\gamma^{2}) \imod{64}.$
Since $\gamma = m + 24n \equiv m \imod{4}$,
$23\gamma^{4} \mp 26\gamma^{3} + 15\gamma^{2} \equiv -\gamma^{4} \mp 2\gamma^{3} - \gamma^{2} \equiv -m^{4}\mp 2m^{3}-m^{2} \equiv 0 \imod{4}.$
It follows that $5 \pm 8A \pm 24A^{3} + 9A^{4} + 18A^{6} \equiv 0 \imod{64}.$ Therefore $J(m, n) \in \Z$. This completes the proof
of Lemma \ref{exlem2}.
\end{proof}
By the discussion in the Remark \ref{remtrans} and \eqref{oldisc1}, it follows that the discriminant of $E_{m, n}$ is
\begin{align}
\Delta_{E_{m, n}} &= \mp \beta(3^{3}\beta^{2} \mp 3^{2}\beta + 1)\nonumber\\
& = \mp 2(m + 24n)(62208n^{2} + (5184m \mp 432)n + (108m^{2} \mp 18m + 1)).\label{egendisc}
\end{align}
Note that from \eqref{nchoice}, all primes $\ell \neq 2$ which divide $\Delta_{E_{m, n}}$ only divide $\Delta_{E_{m, n}}$ once.
To show that $E_{m, n}$ is semistable, we now consider three cases for $m$.

Consider the case when $m = 1, 5, 7, 11, 13, 17, 19, 23$.
These constitute all the odd choices of $m$. Then by the choice of $n$ in \eqref{nchoice},
$\Delta_{E_{m, n}} = \mp 2\delta_{m}(n)$ where
$$\delta_{m}(n) = (m + 24n)(62208n^{2} + (5184m \mp 432)n + (108m^{2} \mp 18m + 1))$$
is squarefree and relatively prime to 2. Therefore $E_{m, n}$ is semistable.

Consider the case when $m = 2, 10, 14, 22$.
These constitute all the even choices for $m$ aside from $m = 8, 16$. Writing $m = 2k$, $k$ odd, we have
$\Delta_{E_{m, n}} = \mp 2^{2}\delta_{m}(n)$ where
$$\delta_{m}(n) = (k + 12n)(62208n^{2} + (10368k \mp 432)n + (432k^{2} \mp 36k + 1)).$$
Since $k$ is odd, $\delta_{m}(n) \equiv 1 \imod{2}$.
By our assumption on $n$, $\delta_{m}(n)$ is squarefree. Thus to show that
$E_{m, n}$ is semistable, it suffices to show that 2 is of multiplicative reduction. Recall that 2 is of multiplicative
reduction if and only if $\ord_{2}(\Delta_{E_{m, n}}) > 0$ (which is clear in this case) and $\ord_{2}(c_{4}') = 0$ (see for example, \cite[Proposition 4.4]{schmitt}).
Since $b_{2} = a_{1}^{2} + 4a_{2}$ and $b_{4} = 2a_{4} + a_{1}a_{3}$, it follows that $c_{4} = b_{2}^{2} - 24b_{4} = 16AD(A^{3}D - 6B)$.
Then as $2^{4}c_{4}' = c_{4}$ and $A \equiv 1 \imod{2}$, we have
\begin{align}\label{c4p}
c_{4}' = AD(A^{3}D - 6B) = A(9A^{3} \pm 8) \equiv 1 \imod{2}
\end{align}
which implies that $\ord_{2}(c_{4}') = 0$. Therefore 2 is of multiplicative reduction for $E_{m, n}$ in this case
and hence $E_{m, n}$ is semistable.

Consider the case when $m = 8, 16$. Writing $m = 8k$ yields that
$\Delta_{E_{m, n}} = \mp 2^{4}\delta_{m}(n)$ where
$$\delta_{m}(n) = (k + 3n)(62208n^{2} + (41472k \mp 432)n + (6912k^{2} \mp 144k + 1)).$$
By our assumption on $n$, $\delta_{m}(n)$ is squarefree and contains at most one factor of 2 (depending on whether
or not $k + 3n$ is even or odd).
Then all primes other than two which divide the discriminant are of multiplicative reduction.
To show that 2 is of multiplicative reduction, it again suffices to show that $\ord_{2}(c_{4}') = 0$
which by the same calculation in the above case shows that this is indeed true. Therefore $E_{m, n}$ is semistable
in this case.
Therefore in all three cases we have shown that $E_{m, n}$ is semistable, satisfying Property \eqref{cond2}.

Since
$\mp 2(m + 24n)(62208n^{2} + (5184m \mp 432)n + (108m^{2} \mp 18m + 1)) \equiv \mp 2m \imod{3}$
and for all $m$ chosen $\mp 2m \not\equiv 0 \imod{3}$, it follows that $E_{m, n}$ satisfies Property \eqref{cond3}.
Furthermore, since $\Delta_{E_{m, n}} = \mp 2^{j}\gamma_{m}(n)$ where $\gamma_{m}(n)$ is squarefree, $(\gamma_{m}(n), 2) = 1$ and $j = 1, 2, 4$, or $5$,
it follows that $E_{m, n}$ satisfies Property \eqref{cond4}.

\begin{rem}
Note that the remaining $m$ not chosen with $1 \leq m \leq 24$ are such that $\mp 2m \equiv 0 \imod{3}$ (occurs when $m = 3, 6, 9, 12, 15, 18, 21, 24$)
and $3 \mid \ord_{2}(\Delta_{E_{m, n}})$ (occurs when $m = 4, 20$).
\end{rem}

Therefore we have shown that $E_{m, n}$ with $m$ and $n$ chosen in \eqref{nchoice} satisfy Properties \eqref{pcond}--\eqref{cond4}.
For each $m$, we claim that there are infinitely many $n$ with $\delta_{m}(n)$ being squarefree which would give us 28 infinite families of elliptic curves.

We consider the case when $m$ is odd, then $i = 1$ and we need to show that there
are infinitely many $n$ such that
\begin{align}\label{oddsqf}
\delta_{m}(n) = (m + 24n)(62208n^{2} + (5184m \mp 432)n + (108m^{2} \mp 18m + 1))
\end{align}
is squarefree. However, since $m$ is fixed, by \cite[Theorem, p. 417]{erdos},
we see that there are indeed infinitely many $n$ for each fixed $m$ such that \eqref{oddsqf}
is squarefree.
The cases when $m$ is even is exactly the same since $\delta_{m}(n)$ in each case is a cubic.
%
%

We have shown the following proposition.
\begin{prop}\label{emnprop}
For $m = 1, 2, 5, 7, 8, 10, 11, 13, 14, 16, 17, 19, 22, 23$, there are infinitely many positive integers $n$
such that $E_{m, n}/\Q$ satisfies Properties \eqref{pcond}--\eqref{cond4}. In particular the desired $n$
are chosen such that $\delta_{m}(n)$ is squarefree.
\end{prop}

\begin{rem}\label{npor}
Not only are there infinitely many such $n$, in fact for each fixed $m$, a positive proportion of $n$ are such that $\delta_{m}(n)$ is squarefree.
Fix an $m$. It can be shown that
$\#\{n \in [1, x]: \delta_{m}(n) \text{ is squarefree}\} \sim cx$ where $c = \prod_{p}(1 - \beta_{m}(p^{2})/p^{2})$
with $\beta_{m}(p^{2}) = \#\{a \imod{p^{2}}: \delta_{m}(a) \equiv 0 \imod{p^{2}}\}$. Further computation yields that
$c\approx 0.85164041$ if $m \neq 8, 16$ and $\approx 0.6387303$ if $m = 8, 16$.
\end{rem}

Since $E_{m, n}$ is semistable, all primes dividing the conductor are of multiplicative reduction. We now show
which ones are split and which ones are nonsplit. This result will play a crucial role in showing that $\om(E_{m, n}^{(d)}) = -1$
for our chosen $d$.
\begin{prop}\label{emnsplit}
Let $\ell$ be a prime of bad reduction for $E_{m, n}$ with $m$ and $n$ chosen as in the discussion around \eqref{nchoice}.
Then $\ell$ is multiplicative reduction and
\begin{enumerate}[(i)]
\item If $\ell = 2$, then $\ell$ is split.
\item If $\ell \mid m + 24n$, then $\ell$ is split.
\item Suppose $\ell \mid 62208n^{2} + (5184m \mp 432)n + (108m^{2} \mp 18m + 1)$. If $\ell \equiv 1 \imod{3}$, then $\ell$ is split.
If instead $\ell \equiv 2 \imod{3}$, then $\ell$ is nonsplit.
\end{enumerate}
\end{prop}
\begin{proof}
Since $E_{m, n}$ is semistable, all primes of bad reduction are of multiplicative reduction.
If $\ell =2 $, then by \cite[Proposition 4.4]{schmitt}, if $x^{2} + a_{1}'x + (a_{3}'a_{1}'^{-1} + a_{2}')$ has a root in $\F_{2}$,
then $2$ is split. Since $a_{1}' = 1$, $a_{2}' = 0$, and $a_{3}' = 0$ and $x^{2} + x$ has both roots in $\F_{2}$, it follows that 2 is split.

If $\ell \neq 2$, then by the same proposition cited above, if $-c_{4}'c_{6}'$ is a square in $\F_{\ell}$, then $\ell$ is split.
If not, then $\ell$ is nonsplit.

Now suppose $\ell \mid m + 24n$. Then $A \equiv \mp 1 \imod{\ell}$. Recall from \eqref{c4p} that $c_{4}' = A(9A^{3} \pm 8)$.
We now compute $c_{6}'$. From the change of variables formula referenced to in Remark \ref{remtrans}, $2^{6}c_{6}' = c_{6}$.
We have $b_{2} = 2^{2}A^{2}D, b_{4} = 2^{2}ABD, b_{6} = 2^{2}B^{2}D$ and hence
\begin{align*}
c_{6} = -b_{2}^{3} + 36b_{2}b_{4} - 216b_{6} = -2^{6}D^{3}A^{6} + 2^{6}3^{2}BD^{2}A^{3} - 2^{5}3^{3}B^{2}D = 2^{6}3^{3}A^{6} \pm 2^{8}3^{2}A^{3} + 2^{9}.
\end{align*}
This implies that
$c_{6}' = 27A^{6} \pm 36A^{3} + 8.$
Then
\begin{align*}
c_{4}'c_{6}' = A(27A^{6} \pm 36A^{3} + 8)(9A^{3}\pm 8) \equiv \mp 1 (27 - 36 + 8)(\mp 1) \equiv -1 \imod{\ell}
\end{align*}
and hence $-c_{4}'c_{6}'$ is a square in $\F_{\ell}$. Thus for all $\ell \mid m + 24n$, $\ell$ is split.

Let $f_{m}(n) := 62208n^{2} + (5184m \mp 432)n + (108m^{2} \mp 18m + 1)$.
We now consider the case when $\ell \mid f_{m}(n)$. Observe that $3f_{m}(n) = A^{2} \mp A + 1$. Then taking both sides
modulo $\ell$ yields that $A^{2} \equiv \pm A -1 \imod{\ell}$.
This implies that modulo $\ell$, we have
\begin{align*}
A^{6} &\equiv (\pm A - 1)^{3} = \pm A^{3} -3A^{2} \pm 3A - 1 \equiv A^{2} \mp A + 2 \equiv 1 \imod{\ell},\\
A^{4} &\equiv (\pm A - 1)^{2} = A^{2} \mp 2A + 1 \equiv \mp A \imod{\ell},
\end{align*}
and
$A^{3} \equiv (\pm A - 1)A \equiv \pm(\pm A - 1) - A = \mp 1 \imod{\ell}.$
Therefore
\begin{align*}
-c_{4}'c_{6}' &= -(9A^{4}\pm 8A)(27A^{6} \pm 36A^{3} + 8) \equiv -(9(\mp A) \pm 8A)(27 \pm 36(\mp 1) + 8) = \mp A \imod{\ell}.
\end{align*}
Since $f_{m}(n) \equiv 0 \imod{\ell}$ and $\mp A = f_{m}(n) - 108(m + 24n)^{2}$,
\begin{align*}
-c_{4}'c_{6}' \equiv \mp A \equiv \mp A - f_{m}(n) = -2^{2}3^{3}(m + 24)^{2} \imod{\ell}
\end{align*}
and hence $-c_{4}'c_{6}'$ is a square in $\F_{\ell}$ if and only if the Legendre symbol $(-3/\ell) = 1$.
Since $(-3/\ell) = 1$ if and only if $\ell \equiv 1 \imod{3}$, it follows that if $\ell \mid f_{m}(n)$ and
$\ell \equiv 1 \imod{3}$, then $\ell$ is split and if instead $\ell \equiv 2 \imod{3}$, then $\ell$ is nonsplit.
This completes the proof of Proposition \ref{emnsplit}.
\end{proof}

\subsection{Torsion Subgroups and Global Root Number}
The following two propositions shows that for certain $m$, $E_{m, n}$ satisfies Assumption \ref{asmp5a}.
\begin{prop}\label{torprop}
Let $m = 1, 2, 5, 7, 8, 10, 11, 13, 14, 16, 17, 19, 22, 23$ and $n$ be defined as in \eqref{nchoice}. If $d$ is a positive squarefree integer
with $d \equiv 1 \imod{12}$, then $E_{m, n}^{(d)}(\Q)_{\textrm{tors}} = 0$.
\end{prop}
\begin{proof}
From the discussion before \eqref{cohomseq}, the image of the mod 3 Galois representation corresponding to $E_{m, n}^{(d)}$ in $\operatorname{GL}_{2}(\F_{3})$ looks
like $\smat{\chi_{d}(\sigma)\om(\sigma)}{\ast}{0}{\chi_{d}(\sigma)}.$
In particular, neither $\chi_{d}\om$ nor $\chi_{d}$ are trivial characters and hence it follows that $E_{m, n}^{(d)}$ has no rational 3-torsion points
(since if $E_{m, n}^{(d)}(\Q)[3] \neq \emptyset$, then top left entry in the above matrix would be 1, however, this is not the case since
neither of $\chi_{d}\om$ nor $\chi_{d}$ are trivial characters).

We now consider $\wt{E}_{m, n}^{(d)}(\F_{3})$. Note that 3 is of good reduction for $E_{m, n}^{(d)}$.
By the general formula for a quadratic twist of an elliptic curve over a field of characteristic 0
(see for example \cite[p. 410]{connell}),
\begin{align}\label{teq}
E_{m, n}^{(d)}/\Q \colon y^{2} + xy = x^{3} + \frac{d - 1}{4}x^{2} + H(m, n)d^{2}x + J(m, n)d^{3}.
\end{align}
Note that the above equation is not the global minimal Weierstrass equation.
Recall from \eqref{mineqn}, Lemma \ref{exlem1}, and Lemma \ref{exlem2} that
$$H(m, n) = \frac{1}{48}(1 \mp 8A - 9A^{4}), \quad J(m, n) = -\frac{1}{2^{6}3^{2}}(5 \pm 8A \pm 24A^{3} + 9A^{4} + 18A^{6}).$$
Since $A \equiv \mp 1 \imod{9}$, $1 \mp 8A \equiv 0 \imod{9}$ and hence
$$2^{4}\cdot 3H(m, n) = 1 \mp 8A - 9A^{4} \equiv 0 \imod{9}$$ which implies that $H(m, n) \equiv 0 \imod{3}$.
Setting $\alpha := 2\cdot 3^{2}(m + 24n)$, from \eqref{alphacalc},
$$-2^{6}3^{2}J(m, n) = 5 \pm 8A \pm 24A^{3} + 9A^{4} + 18A^{6} \equiv \mp 2^{6}\alpha \equiv \mp 2^{7}3^{2}(m + 24n) \imod{27}.$$
This implies that $J(m, n) \equiv \pm 2m \imod{3}.$
Therefore depending on $\pm 2m \imod{3}$, we have two cases, when $J(m, n) \equiv 1 \imod{3}$ or when $J(m, n) \equiv 2 \imod{3}$.
As $d \equiv 1 \imod{12}$, we have
\begin{align}\label{redeq}
\wt{E}_{m, n}^{(d)}/\F_{3} \colon y^{2} + xy = x^{3} + J(m, n).
\end{align}
First, suppose $J \equiv 1 \imod{3}$. From \eqref{redeq}, we have that the reduced equation is $y^{2} + xy = x^{3} + 1$.
One can compute that $\wt{E}_{m, n}^{(d)}(\F_{3}) = \Z/6\Z$ and hence $\#\wt{E}_{m, n}^{(d)}(\F_{3}) = 6$.
Second, suppose $J \equiv 2 \imod{3}$. From \eqref{redeq}, we have that the reduced equation is $y^{2} + xy = x^{3} + 2$.
One can compute that $\wt{E}_{m, n}^{(d)}(\F_{3}) = \Z/3\Z$ and hence $\#\wt{E}_{m, n}^{(d)}(\F_{3}) = 3$.
Since $E_{m, n}^{(d)}$ has good reduction modulo 3, $E_{m, n}^{(d)}(\Q)_{\textrm{tors}}$ injects into $\wt{E}_{m, n}^{(d)}(\F_{3})$
and since $E_{m, n}^{(d)}/\Q$ has no rational 3-torsion, if $J \equiv 2 \imod{3}$, $E_{m, n}^{(d)}(\Q)_{\textrm{tors}} = 0$.

If $J \equiv 1 \imod{3}$, to show that $E_{m, n}^{(d)}(\Q)_{\textrm{tors}} = 0$, it suffices to show that $E_{m, n}^{(d)}/\Q$ has
no rational 2-torsion points. To do this, we use the 2-division polynomial for $E_{m, n}^{(d)}$.
Suppose $E_{m, n}^{(d)}/\Q$ had a rational 2-torsion point $(x, y)$. Then $\psi_{2}(x, y) = 2y + x = 0$ and hence $x = -2y$.
Using the expression in \eqref{teq} yields that $y$ must satisfy
\begin{align}\label{ypol}
8y^{3} - dy^{2} + 2H(m, n)d^{2}y - J(m, n)d^{3} = 0.
\end{align}
However, we claim that \eqref{ypol} is irreducible over $\Q$. Indeed, computing the discriminant of the polynomial
on the left hand side yields
\begin{align}\label{ypdisc}
\mp 8d^{6}(m + 24n)(62208n^{2} + (5184m \mp 432)n + (108m^{2} \mp 18m + 1)).
\end{align}
As $n$ is chosen such that $(m + 24n)(62208n^{2} + (5184m \mp 432)n + (108m^{2} \mp 18m + 1))$ is squarefree,
it follows that \eqref{ypdisc} is not a square in $\Q$. Therefore \eqref{ypol} has Galois group $S_{3}$ and is irreducible
over $\Q$. Therefore no such $y$ exists and hence $E_{m, n}^{(d)}/\Q$ has no rational 2-torsion point.
It follows that $E_{m, n}^{(d)}(\Q)_{\textrm{tors}} = 0$ in this case.

This completes the proof of Proposition \ref{torprop}.
\end{proof}

\begin{prop}\label{signprop}
Let $m = 1, 7, 13, 19$ and $d$ be chosen as in the discussion around \eqref{dcond} and \eqref{dtab2}. Then the global root number
$\om(E_{m, n}^{(d)}) = -1$.
\end{prop}
\begin{proof}
By the discussion around Assumption \ref{asmp1}, recall that $E_{m, n}^{(d)}$  has additive reduction at all primes dividing $d$.
As $d \equiv 1 \imod{4}$, there are an even number of primes $p \mid d$ which are $3 \imod{4}$.
Since each such $p$ contributes a factor of $(-1/p)$ to $\om(E_{m, n}^{(d)})$ and nonsplit multiplicative primes of $E_{m, n}^{(d)}$
contribute a factor of 1 to $\om(E_{m, n}^{(d)})$, it follows that
$\om(E_{m, n}^{(d)}) = -(-1)^{S}$ where $S$ is the number of primes of split multiplicative reduction for $E_{m, n}^{(d)}$.
It suffices to show that for the chosen $m$, $E_{m, n}^{(d)}$ has an even number of primes of split multiplicative reduction.

Let $\ell$ be a prime of multiplicative reduction. Then by the discussion around Assumption \ref{asmp1}, $\ell \mid N_{E_{m, n}}$.
By Proposition \ref{emnsplit}, Proposition \ref{prop1}, and the choice of $d$ in \eqref{dtab2}:
\begin{enumerate}[($a$)]
\item If $\ell = 2$, then $2$ is of split multiplicative reduction for $E_{m, n}$. Since $d \equiv 5 \imod{8}$, $d$ is not a square
in $\Q_{2}$ and hence $2$ is of nonsplit multiplicative reduction for $E_{m, n}^{(d)}$.
\item If $\ell \mid m + 24n$, then $\ell$ is of split multiplicative reduction for $E_{m, n}$.
Observe that as $m = 1, 7, 13, 19$, $m + 24n \equiv 1 \imod{3}$, $\ell \geq 5$ and there are an even number of primes which are $2 \imod{3}$ and divide $m + 24n$.
If $\ell \equiv 1 \imod{3}$, then $d \not\equiv \square\imod{\ell}$
and hence $\ell$ is of nonsplit multiplicative reduction for $E_{m, n}^{(d)}$. If $\ell \equiv 2 \imod{3}$, then $d \equiv 1 \imod{\ell}$
and hence $\ell$ is of split multiplicative reduction for $E_{m, n}^{(d)}$.
\item If $\ell \mid 62208n^{2} + (5184m \mp 432)n + (108m^{2} \mp 18m + 1)$, we consider when $\ell \equiv 1 \imod{3}$ and when $\ell \equiv 2 \imod{3}$.
Since
$$62208n^{2} + (5184m \mp 432)n + (108m^{2} \mp 18m + 1) \equiv 1 \imod{6},$$
$\ell \neq 2$ and there are an even number of prime factors which are $2 \imod{3}$.
If $\ell \equiv 1 \imod{3}$, then $\ell$ is of split multiplicative reduction for $E_{m, n}$ and hence by our choice of $d$, $\ell$ is of nonsplit multiplicative reduction
for $E_{m, n}^{(d)}$. If $\ell \equiv 2 \imod{3}$, then $\ell$ is of nonsplit multiplicative reduction for $E_{m, n}$ and hence by our choice of $d$, $\ell$ is of split
multiplicative reduction for $E_{m, n}^{(d)}$.
\end{enumerate}
Therefore $E_{m, n}^{(d)}$ has an even number of primes of split multiplicative reduction. This implies that
$\om(E_{m, n}^{(d)}) = -1$. This completes the proof of Proposition \ref{signprop}.
\end{proof}

The proof of Theorem \ref{exthm} is now immediate.
\begin{proof}[Proof of Theorem \ref{exthm}]
To prove Theorem \ref{exthm}, it suffices to verify that $E_{m, n}/\Q$, $m = 1$, $7$, $13$, $19$
satisfies Assumptions \ref{asmp1}--\ref{asmp3}, \ref{asmp5a}, \ref{asmp5} (note that Assumptions \ref{asmp4a} and \ref{asmp4} do not deal
with the elliptic curve).
By Property \eqref{cond2}, $E_{m, n}$ is semistable and satisfies Assumption \ref{asmp1}.
Since $m$ is odd, by the paragraph immediately following \eqref{egendisc},
$\Delta_{E_{m, n}} = \mp 2\delta_{m}(n)$ where $\delta_{m}(n)$ is squarefree and relatively
prime to 2. Therefore $\ord_{2}(\Delta_{E_{m, n}}) = 1$ and in particular is odd.
This and as $E_{m, n}$ satisfy Properties \eqref{cond3} and \eqref{cond4} imply
that $E_{m, n}$ satisfies Assumption \ref{asmp2}.
By construction, $E_{m, n}$ satisfies Property \eqref{pcond} and hence
satisfies Assumption \ref{asmp3}.
Propositions \ref{torprop} and \ref{signprop} (along with the fact that
$d \equiv 1 \imod{12}$ by Assumptions \ref{asmp1} and \ref{asmp4a})
imply that $E_{m, n}$ satisfies Assumption \ref{asmp5a}.
Finally, by Proposition \ref{emnsplit}, $\ord_{2}(N_{E_{m, n}}) = 1$
and hence Assumption \ref{asmp5} is satisfied.
This completes the proof of Theorem \ref{exthm}.
\end{proof}

Applying Corollary \ref{rescor} to the elliptic curves $E_{m, n}$, $m = 1, 7, 13, 19$ shows that assuming the rank
part of the Birch and Swinnerton-Dyer conjecture, each $E_{m, n}$ is such that a positive proportion of its quadratic
twists have rank 1.